\documentclass[12pt]{amsart}
\usepackage[english]{babel}
\usepackage[matrix,arrow]{xy}
\usepackage{graphicx}
 \usepackage{multicol,combelow,mdframed}
\usepackage{enumerate}
\usepackage{amsfonts,mathrsfs,amssymb,amstext,amsmath}
\usepackage{todonotes}
\usepackage{latexsym}
\usepackage{color,float}
\usepackage{fancyhdr}
\usepackage{geometry}
\usepackage{soul}
\usepackage{tikz,tasks}
\usepackage[colorlinks=true,linkcolor=blue,citecolor=blue]{hyperref}
\RequirePackage{fix-cm}
\numberwithin{equation}{section}
\allowdisplaybreaks

\newtheorem{theorem}{Theorem}[section]
\newtheorem{lemma}[theorem]{Lemma}
\newtheorem{proposition}[theorem]{Proposition}

\newtheorem{remark}[theorem]{Remark}
\newtheoremstyle{example}{}{}{}{}{}{}{}{}
\newtheorem*{definition*}{Definition}
\newtheorem*{theorem*}{Theorem}
\newtheorem*{lemma*}{Lemma}
\newtheorem*{proposition*}{Proposition}

\newtheorem{example}[theorem]{Example}
\newtheorem{definition}[theorem]{Definition}

\newcommand{\diffto}{\xrightarrow{\raisebox{-0.2 em}[0pt][0pt]{\smash{\ensuremath{\sim}}}}}

\newcommand{\rmap}{\longrightarrow}

\newcommand{\Iff}{\Longleftrightarrow}
\newcommand{\acts}{\curvearrowright}
\newcommand{\la}{\langle}
\newcommand{\ra}{\rangle}

\newcommand{\id}{\operatorname{id}}
\newcommand{\pr}{\operatorname{pr}}
\newcommand{\X}{\mathfrak{X}}
\newcommand{\bd}{\partial}

\newcommand{\dd}{\mathrm{d}}

\newcommand{\R}{\mathbb{R}}

\newcommand{\C}{\mathbb{C}}

\newcommand{\abk}{\langle\cdot,\cdot\rangle}
\newcommand{\sbk}{[\cdot,\cdot]}

\begin{document}
	
\title{Coregular submanifolds and Poisson submersions}
\author[L. Brambila]{Lilian Cordeiro Brambila}
\address{Departamento de Matem\'{a}tica, Universidade Federal do Paran\'{a}, DAMAT-CT, UTFPR}
\email{lilianc@utfpr.edu.br}
\author[P. Frejlich]{Pedro Frejlich}
\address{Departamento de Matem\'atica Pura e Aplicada, IME, UFRGS}
\email{frejlich.math@gmail.com}
\author[D. Mart\'inez Torres]{David Mart\'{i}nez Torres}
\address{Departamento de Matem\'{a}tica Aplicada, ETSAM, UPM}
\email{df.mtorres@upm.es}
\begin{abstract}
We analyze \emph{submersions with Poisson fibres}. These are submersions whose total space carries a Poisson structure, on which the ambient Poisson structure pulls back, as a Dirac structure, to Poisson structures on each individual fibre. Our ``Poisson-Dirac viewpoint’’ is prompted by natural examples of Poisson submersions with Poisson fibers --- in toric geometry and Poisson-Lie groups ---  whose analysis was not possible using the existing tools in the Poisson literature.

The first part of the paper studies the Poisson-Dirac perspective of inducing Poisson structures on submanifolds. This is a rich landscape, in which subtle behaviours abound --- as illustrated by a surprising ``jumping phenomenon'' concerning the complex relation between the induced and the ambient symplectic foliations, which we discovered here. These pathologies, however, are absent from the well-behaved and abundant class of \emph{coregular} submanifolds, with which we are mostly concerned here.

The second part of the paper studies Poisson submersions with Poisson fibres --- the natural Poisson generalization of flat symplectic bundles. These Poisson submersions have coregular Poisson-Dirac fibres, and behave functorially with respect to such submanifolds. We discuss the subtle collective behavior of the Poisson fibers of such Poisson fibrations, and explain their relation to pencils of Poisson structures.

The third and final part applies the theory developed to Poisson submersions with Poisson fibres which arise in Lie theory. We also show that such submersions are a convenient setting for the associated bundle construction, and we illustrate this by producing new Poisson structures with a finite number of symplectic leaves. 

Some of the points in the paper being fairly new, we illustrate the many fine issues that appear with an abundance of (counter-)examples.
\end{abstract}

\maketitle

\setcounter{tocdepth}{1}
 \tableofcontents

\section{Introduction}

The condition that a submanifold of a Poisson manifold  ``inherit'' a Poisson structure is somewhat subtle, since bivectors cannot be pulled back as such. However, we understand since the work of Ted Courant on Dirac structures \cite{Cou} that there is a \underline{canonical  candidate} to an ``induced Poisson structure". Indeed, a Poisson structure $\pi_M$ on $M$ may be regarded as a Dirac structure via its graph:
\[
\mathrm{Gr}(\pi_M) = \{ \pi_M^{\sharp}(\xi)+\xi  \ | \  \xi \in  T^*M\} \subset TM \oplus T^*M=\colon\mathbb{T}M
\]

It turns out that every submanifold $X$ inherits a canonical (pullback) Lagrangian family $i^!\mathrm{Gr}(\pi_M) \subset \mathbb{T}X$ (see Section \ref{sec : Poisson-Dirac submanifolds} for more details), and \underline{at most} one bivector $\pi_X \in \X^2(X)$ can exist on $X$ for which
\begin{align*}
 \mathrm{Gr}(\pi_X)=i^!\mathrm{Gr}(\pi_M).
\end{align*}
When that is the case, $\pi_X$ is automatically Poisson, and we say that $(X,\pi_X)$ is a {\bf Poisson-Dirac submanifold} of $(M,\pi_M)$. This notion was first introduced in \cite[Section 8]{CF}.

This is the most general recipe  to ``induce''
Poisson structures on submanifolds, but it is rather difficult to check, in that the existence and smoothness of $\pi_X$ are \emph{assumed} as opposed to \emph{deduced} from some more straightforwardly verifiable condition, and deciding whether $\pi_X$ is smooth  or not may be technically challenging in practice.

In the literature concerning submanifolds of a Poisson manifold, it was (mistakenly) thought that the induced symplectic foliation on a Poisson-Dirac submanifold would behave well with respect to the ambient symplectic foliation --- in the sense that the leaves of the Poisson-Dirac submanifold would arise as (clean) intersections of the submanifold with the ambient leaves. For example, Poisson-Dirac submanifolds and their clean (or split) counterparts are confused in:
\begin{itemize}
\item The foundational paper \cite{CF} (cf. \cite[Definition 4]{CF} and \cite[Corollary 10]{CF});
\item In \cite{Za} (cf. \cite[Definition 5]{Za} and \cite[Definition 6]{Za});
\item In \cite{Paula} (see \cite[Lemma A.1]{Paula});
\item In \cite{Arathoon} (see \cite[Theorem 2.1]{Arathoon}).
\end{itemize}

It turns out, however, that the symplectic foliation of the Poisson-Dirac submanifold can be \ul{wildly} different from that of the ambient manifold. In fact, the first contribution in this note is the discovery of a
\vspace{0.1cm}
\begin{mdframed}
{\bf Jumping phenomenon: }A leaf of a Poisson-Dirac submanifold need not lie inside a leaf of the ambient manifold;
\end{mdframed}
\vspace{0.1cm}
\noindent see Examples \ref{ex : intersections not manifolds} and \ref{ex : jumping phenomenon}.\\

One example of a condition on a submanifold $X$ of a Poisson manifold $(M,\pi_M)$, which turns out to be easy to check, and ensures that $X$ has the structure of clean Poisson-Dirac submanifold, is to require that $X$ be {\bf split} --- that is, that there exist a splitting $TM|_X=TX\oplus E$ with $\pi_M|_X = \pi_X$ modulo $\Gamma(\wedge^2E)$. This condition turns out to be equivalent to demanding that Hamiltonian flows of $(X,\pi_X)$ be the restriction of Hamiltonian flows of $(M,\pi_M)$ --- otherwise said, split manifolds are those Poisson-Dirac submanifolds in which the jumping phenomenon alluded to above is ruled out by fiat.

{\bf Coregular} Poisson-Dirac submanifolds, first introduced by Courant \cite[Theorem 3.2.1]{Cou}, are a particular case of split submanifolds. A submanifold $X$ is a coregular Poisson-Dirac submanifold if $TX$ and the image $\pi_M^{\sharp}(N^*X)$ of the conormal bundle of $X$  meet trivially, and their direct sum $TX \oplus \pi_M^{\sharp}(N^*X)$ is a vector subbundle of $TM|_X$. Said otherwise, these are split Poisson-Dirac submanifolds, in which the projection to the normal bundle of $X$
\begin{align*}
 & \mathrm{Q}_X \colon N^*X \to NX, & \mathrm{Q}_X(\xi) = \pi_M^{\sharp}(\xi)+TX
\end{align*}has constant rank.  Coregular submanifolds comprise, among others:
\begin{itemize}
    \item {\bf Poisson submanifolds:} submanifolds $X$ to which every Hamiltonian vector field of $(M,\pi_M)$ is tangent; that is,
    \begin{align*}
        \pi_M^{\sharp}(T^*M|_X) \subset TX ;
    \end{align*}
    \item {\bf Poisson transversals:} submanifolds $X$ which meet each leaf of $(M,\pi_M)$ transversally and symplectically; that is,
    \begin{align*}
        TM|_X = TX \oplus \pi_M^{\sharp}(N^*X);
    \end{align*}
    \item Any {\bf point} in a Poisson manifold.
\end{itemize}

We devote Section \ref{sec : Poisson-Dirac submanifolds} to the a detailed description of this whole hierarchy of regularity conditions
\begin{figure}[H]
 \centering
 \includegraphics[keepaspectratio=true,scale=0.8]{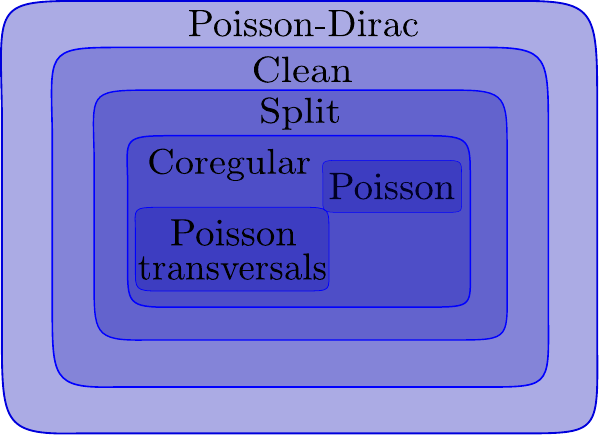}
 \caption{Hierarchy of induced Poisson structures}
\label{fig:hierarchyintro1}
\end{figure}
\noindent with a heavy emphasis on (counter-)examples. This complements (and in some cases, corrects) the endeavours in \cite{CF},  \cite{MeP}, \cite{PT1}, \cite{Xu}, and to some extent those in \cite{CZ}, \cite{Za} (which discuss distinguished submanifolds of Poisson manifolds which do not necessarily inherit Poisson structures). \\

 The main object of interest in the present paper are submersions
\begin{align*}
 p \colon (\Sigma,\pi_{\Sigma}) \rmap (M,\pi_M)
\end{align*}between Poisson manifolds. The natural ``compatibility'' condition one can impose is:
\begin{enumerate}[a)]
 \item that $p$ be a Poisson map.
\end{enumerate}
\begin{align*}
 & \{ f \circ p,g \circ p\}_{\Sigma} = \{f,g\}_M \circ p, & f,g \in C^{\infty}(M).
\end{align*}
This is the correct Poisson-theoretic notion to ensure that $p$ intertwines Hamiltonian flows of $f$ and $f \circ p$. Note, however, that this does not imply that $p$ intertwines the induced symplectic foliations
on the total space and base --- in the sense that the preimage of leaves is saturated. For example, symplectic realizations exist for any Poisson manifold \cite{Ka,We}, and those only intertwine symplectic foliations in this sense if $(M,\pi_M)$ is symplectic.\\

In light of the discussion on submanifolds, another natural ``compatibility'' condition one can impose is:

\begin{enumerate}[b)]
 \item that fibres of $p$ have an induced Poisson structure.
\end{enumerate}

\begin{definition*}
 A submersion $p \colon (\Sigma,\pi_{\Sigma}) \to M$ from a Poisson manifold $(\Sigma,\pi_{\Sigma})$ has {\bf Poisson fibres} if each of its fibres is a Poisson-Dirac submanifold.
\end{definition*}

Submersions with Poisson fibres encompass previously studied classes (such as vertical, coupling \cite{BF,Vo,Wade} and almost-coupling Poisson structures \cite{BF,Va}). All these cases assume a good collective behaviour of the Poisson-Dirac fibers --- local triviality or, at the very least, the existence of a compatible Ehresmann connection. Our Poisson-Dirac standpoint makes no such assumptions.

Of great importance to what follows is that \ul{Poisson submersions} with Poisson fibres
\[
p \colon (\Sigma,\pi_{\Sigma}) \rmap (M,\pi_M)
\]have in fact coregular fibres, and behave functorially in the sense below:

\begin{theorem*}
A Poisson submersion $p \colon (\Sigma,\pi_{\Sigma}) \to (M,\pi_M)$ has Poisson fibres exactly when each vertical tangent space inherits a Poisson structure. In that case, its fibres are coregular. 

Moreover, for a coregular Poisson-Dirac submanifold $Y \subset M$,
  \begin{enumerate}[a)]
   \item $X\colon=p^{-1}(Y) \subset \Sigma$ is a coregular Poisson-Dirac submanifold;
   \item $p \colon(X,\pi_X) \to (Y,\pi_Y)$ is a Poisson submersion with Poisson fibres.
  \end{enumerate}
\end{theorem*}

Such good functorial behavior of Poisson submersion with Poisson fibres, as opposed to general Poisson 
submersions, has implications regarding the symplectic foliations of the total space and base.

\begin{theorem*}
 A Poisson submersion with Poisson fibres $p\colon(\Sigma,\pi_{\Sigma}) \to (M,\pi_M)$ pulls symplectic leaves of 
 $M$ back to Poisson submanifolds of $\Sigma$, over which $p$ restricts to coupling Poisson submersions.
\end{theorem*}

 These results convey a clear picture of a Poisson submersion with Poisson fibres
\[
p\colon(\Sigma,\pi_{\Sigma}) \to (M,\pi_M):
\]it maps leaves of $\pi_{\Sigma}$ into leaves of $\pi_M$, in which case the ensuing restrictions between symplectic leaves
\[
p\colon\mathrm{S}_{\Sigma}(x) \to \mathrm{S}_{M}(p(x))
\]are \emph{flat symplectic bundles}. So, very much like a Poisson structure on a manifold makes precise the idea of a assembling symplectic manifolds, a Poisson submersion with Poisson fibres makes precise the idea of assembling flat symplectic bundles.

In fact, the good  behavior with respect to the symplectic foliations --- or even with respect to their distributions --- \ul{characterizes} Poisson submersion with Poisson fibres among Poisson submersions:
\begin{theorem*} The following statements for a Poisson submersion $p:(\Sigma,\pi_{\Sigma}) \to (M,\pi_M)$ are equivalent:
 \begin{enumerate}[(a)]
  \item It is a Poisson submersion with Poisson fibres.
  \item It maps symplectic leaves of $(\Sigma,\pi_{\Sigma})$ into symplectic leaves of $(M,\pi_M)$.
  \item Its differential maps the characteristic distribution of $(\Sigma,\pi_{\Sigma})$ into the characteristic distribution of $(M,\pi_M)$.
 \end{enumerate}
\end{theorem*}

 Poisson submersions with Poisson fibres have their own version of the jumping phenomenon. A Poisson submersion with Poisson fibres over a symplectic manifold (that is, a coupling) induces, under a mild completeness assumption, \emph{diffeomorphic} Poisson structures on its fibres. However:

 \vspace{0.5cm} 
 
\begin{mdframed}
\noindent {\bf Jumping phenomenon:}  If the base manifold of Poisson submersions with Poisson fibres is not symplectic, the Poisson diffeomorphism type of fibres may very well \ul{change} as we change symplectic leaves in the base --- remarkably, they need not even assemble into a vertical Poisson structure.
\end{mdframed}

\vspace{0.5cm} 

\noindent In fact, for a Poisson submersion with Poisson fibres $p\colon(\Sigma,\pi_{\Sigma}) \to (M,\pi_M)$
\vspace{-0.3cm} 
\begin{theorem*}
 The Poisson fibres assemble into a vertical Poisson structure $\pi_V \in \Gamma(\wedge^2V)$ exactly when a Poisson structure $\pi_H \in \X^2(\Sigma)$ exists, whose symplectic foliation arises from Hamiltonian flows of functions on $M$. When that is the case, $\pi_V$ and $\pi_H$ commute and split $\pi_{\Sigma}$:
 \begin{align*}
     & \pi_{\Sigma}=\pi_V+\pi_H, & [\pi_H,\pi_V]=0.
 \end{align*}
\end{theorem*}

\noindent We refer to such objects as {\bf orthogonal pencils}, to highlight the somewhat surprising fact that the Poisson structure on the total space arises from a pencil of Poisson structures.

Yet again, the landscape is fraught with subtleties (e.g., orthogonal pencils need not be almost-coupling), and we devote the bulk of the Section \ref{sec : coregular Poisson submersions} to lay out the foundational aspects of the theory of Poisson submersions with Poisson fibres, striving to offer as many examples as a clear picture requires. Just as for submanifolds, our analysis yields a whole hierarchy of Poisson submersions
\begin{figure}[H]
 \centering
 \includegraphics[keepaspectratio=true,scale=0.8]{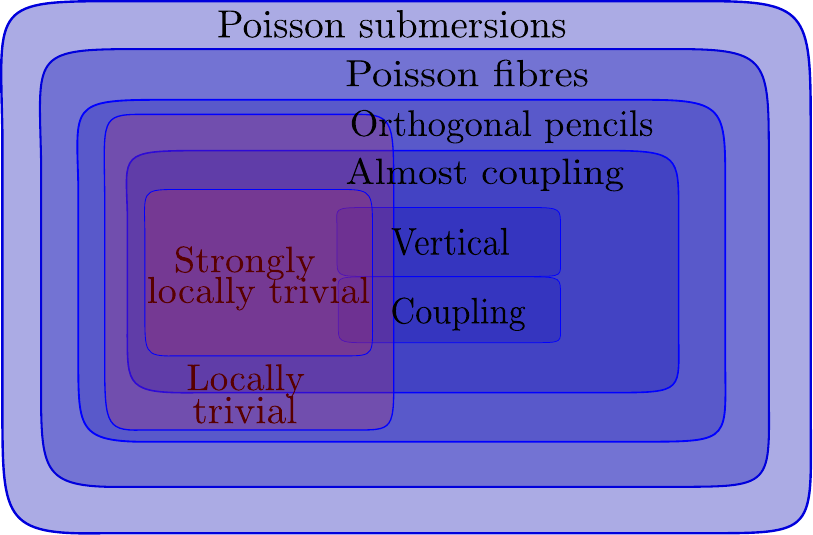}
 \caption{Hierarchy of Poisson submersions}
 \label{fig:hierarchyintro2}
\end{figure}

The remainder of the paper is devoted to applications and examples of Poisson submersions with Poisson fibres coming from Lie theory. The pointwise-to-global feature of coregular manifolds manifests itself in the following key:

\begin{lemma*}\normalfont
Let a Poisson-Lie group $(G,\pi_G)$ have a Poisson action on a Poisson manifold $(M,\pi_M)$. Then an orbit $X$ of $G \curvearrowright M$ is a coregular Poisson-Dirac submanifold iff $T_xX \subset (T_xM,\pi_{M,x})$ has an induced Poisson structure for some $x \in X$.
\end{lemma*}

Applied to the action of a complex vector space $A$ on a complex manifold $M$ by holomorphic transformations, there corresponds to each {\bf positive} bivector $\pi_A \in \wedge^2A$ --- that is, one whose leaves are K\"ahler manifolds --- an induced $A$-invariant Poisson  structure $\pi_M$ on $M$, such that every complex subspace of $A$ acts on
 $(M,\pi_M)$ by Poisson diffeomorphisms, and has coregular Poisson-Dirac orbits (Lemma \ref{lem : holomorphic actions}). One instance is the GIT presentation of a toric manifold $M_{\varDelta}$ --- a certain principal bundle $p\colon\Sigma_{\varDelta} \to M_{\varDelta}$ with structure group a complex torus, and total space $\Sigma_{\varDelta}$ an open set in $\mathbb{C}^d$, constructed out of a Delzant polytope $\varDelta$: 

\begin{proposition*}
Every positive bivector $\pi \in \wedge^2 \C^d$ turns the GIT presentation $p:\Sigma_{\varDelta} \to M_{\varDelta}$ of the toric variety $M_{\varDelta}$ into a Poisson submersion with Poisson fibres.
\end{proposition*}Following \cite{Ca}, we call those Poisson structures which arise from nondegenerate, positive bivectors (and which 
therefore have finitely many leaves) {\bf toric Poisson manifolds}. \\

The second class of examples concerns the ``standard'' (or Lu-Weinstein) Poisson-Lie group structure $\pi_G$ on a compact connected semisimple Lie group $G$ associated with a choice of maximal torus $T$ (and root ordering). This structure --- discovered in \cite{LW} --- descends to a ``standard'' Poisson structure $\pi_M$ on the manifold of full flags $M=G/T$, with finitely many leaves, the Bruhat cells:

\begin{proposition*}
 The quotient map $p:(G,\pi_G) \to (M,\pi_M)$ is a Poisson submersion with Poisson fibres, inducing the trivial Poisson structure on fibres.
\end{proposition*}

We conclude the paper with a discussion of the associated bundle construction in 
the context of Poisson submersions. Namely, a right principal $G$-bundle $p\colon P \to M$ 
equipped with a $G$-invariant Poisson structure $\pi_P$ determines a Poisson submersion
\[
 p \colon(P,\pi_P) \to (M,\pi_M).
\]If $G$ acts on the left of $(X,\pi_X)$ by Poisson diffeomorphisms, the associated bundle $\Sigma \colon=P \times_GX$ has an induced Poisson structure $\pi_{\Sigma}$, for which
\[
 p: (\Sigma,\pi_{\Sigma}) \to (M,\pi_M)
\]is again a Poisson submersion.
 We will see that good behavior of the principal Poisson submersion $p:(P,\pi_P) \to (M,\pi_M)$ is inherited by the associated bundle $p\colon (\Sigma,\pi_{\Sigma}) \to (M,\pi_M)$  (Lemma \ref{lem : associated inherits pencils}), but also that conditions can be imposed on $\pi_P$ and the action $G \curvearrowright (X,\pi_X)$ to ensure that the associated bundle behaves well even if the principal bundle itself does not (Lemma \ref{lem : self-made pencil}).

In the study of such ``Poisson associated bundles'', the need arises to impose some 
condition akin to local triviality. When the base of the submersion is symplectic, local triviality has a transparent, mandated meaning (see \cite{Fr}), but its meaning for general  Poisson submersions with Poisson fibres is less clear. For our purposes, it suffices to consider three notions, corresponding the the following local models of $\pi_{\Sigma}$: the product of Poisson structures on the base and on the fibre (\emph{strongly locally trivial}), its gauge-transform by a closed two-form (\emph{locally trivial}), and having the same singular foliation as those (\emph{locally trivial foliation}). The delicate issue will be that this hierarchy of local triviality notions will not always pass from principal to associated bundles. However:

 \begin{proposition*}
The associated bundle has locally trivial foliation, provided the orbits of $G \curvearrowright (X,\pi_X)$ lie inside symplectic leaves, and either:
\begin{enumerate}[a)]
\item The Poisson structures on the fibers of $p:P \to M$ are all trivial, or
\item The orbits of $G\curvearrowright X$ are isotropic submanifolds of the symplectic leaves of $(X,\pi_X)$.
\end{enumerate}
 \end{proposition*}
 
 \noindent In fact, in that case the leaf spaces of $(\Sigma,\pi_{\Sigma})$ and $(M,\pi_M) \times (X,\pi_X)$ are homeomorphic. This leads to the last application of the paper: the construction of Poisson submersions with Poisson fibres with finitely many leaves, using as models for base/fibres toric Poisson manifolds or manifolds of full flags.\\

\noindent {\bf Conventions} By a \emph{singular foliation} $\mathcal{S}$ on a smooth manifold $M$ we mean a partition of $M$ into initial, immersed submanifolds, in the sense of Stefan-Sussmann \cite{St,Su}. We note that all 
singular foliations in this paper will be the orbit partition of a Lie algebroid \cite{AS}. 

 A \emph{foliation} $\mathscr{F}$, tout court, is a singular foliation by equidimensional leaves (that is, an honest foliation in the usual sense). \emph{Submanifold} is to be understood as an embedded submanifold --- but in the Appendix we consider the case of embedded submanifolds whose connected components may have different dimensions. 

For a Poisson manifold $(M,\pi_M)$, we denote by $\mathrm{H}_f\colon=\pi_M^{\sharp}(\mathrm{d}f)$ the Hamiltonian vector field of a (possibly time-dependent) function $f$ on $M$, and by $\phi_{\mathrm{H}_f}^{t,s}$ its local flow. 

The symplectic leaf through $x \in M$ (the orbit through Hamiltonian local flows) is denoted by $\mathrm{S}_M(x)$, and $\mathcal{S}_M$ stands for the ensuing singular foliation of $M$. In all cases but one, the singular foliations considered in the paper will be of this kind (see e.g. \cite[Section 4.1]{CFM}). 

For references on the basics of Poisson geometry, with a viewpoint similar to the one espoused here, we also recommend \cite{MeD}, \cite{LPV} and \cite{MeP}. For those on Dirac structures, see e.g. \cite{Bu} and \cite{Gu}, and for Lie group structure theory, see \cite{Kn}. \\

\noindent {\bf Terminology} The terminology concerning submanifolds of Poisson manifolds is quite inconsistent, and we have adopted our own. One reason for this inconsistency was the hitherto unknown jumping phenomenon we discuss in the paper.  For instance:

--- ``Poisson-Dirac submanifolds" in our sense are those in \cite[Proposition 6]{CF} and \cite[Definition 6]{Za};

--- What we call ``clean Poisson-Dirac submanifolds" are \emph{also} called ``Poisson-Dirac submanifolds" in the same works \cite[Definition 4]{CF} and \cite[Definition 5]{Za} --- the jumping phenomenon was missing from the picture, so both classes were mistakenly conflated;

--- What we call ``split Poisson-Dirac submanifolds" were called ``Poisson-Dirac submanifolds with a Dirac splitting" in \cite[Proposition 7]{CF}, but merely called "Poisson-Dirac submanifolds" in \cite{LPV} and in \cite{MeP};

--- What we call ``coregular submanifold" appears without name in \cite[Theorem 3.2.1]{Cou} and as ``Poisson-Dirac submanifolds of constant rank" in \cite[p. 119]{CF}, and in \cite{CZ} as ``pre-Poisson Poisson-Dirac submanifold". Our terminology, however, agrees with that of \cite{CFM} and that of \cite{Geudens} in the case where the induced Dirac structure is in fact Poisson. Our choice of names seeks to accurately reflect the properties which characterize each of the classes in the hierarchy of Poisson-Dirac submanifolds.\\

\noindent {\bf Acknowledgements.} LB  was partially supported by CAPES 001. DMT was 
partially supported by FAPERJ E‐26/202.923/2015, FAPERJ E-26/010.001249/2016, FAPERJ  E‐26/202.908/2018, and CNPQ 304049/2018-2.

\section{Submanifolds}\label{sec : Poisson-Dirac submanifolds}

Recall that the \emph{generalized tangent bundle} $\mathbb{T}M \colon=TM \oplus T^*M$ carries the canonical symmetric pairing
\begin{align}\label{eq : symmetric bilinear pairing}
 & \abk \colon \mathbb{T}M \times_M \mathbb{T}M, & \langle u+\xi,v+\eta\rangle \colon = \iota_v\xi+\iota_u\eta,
\end{align}and the Dorfman bracket on the space of sections:
\begin{align}\label{eq : Dorfman}
 & \sbk \colon \Gamma(\mathbb{T}M) \times \Gamma(\mathbb{T}M) \rmap \Gamma(\mathbb{T}M), & [u+\xi,v+\eta] \colon =[u,v]+\mathscr{L}_u\eta-\iota_v\dd\xi.
\end{align}
A subspace $L_M \subset \mathbb{T}M$ is a {\bf Lagrangian family} if it meets each fibre $\mathbb{T}_xM$ in a Lagrangian vector subspace. No continuity is assumed for the map $x \mapsto L_{M,x}$; for example,

\begin{example}\label{ex : discontinuous Lagrangian family on R}
    The subspace
    \begin{align*}
    & L \subset \mathbb{T}\mathbb{R}, & L_t = \begin{cases}
        T_t\mathbb{R}, & \text{if} \ t \neq 0,\\
        T^*_0\mathbb{R}, & \text{if} \ t = 0
    \end{cases}
\end{align*}is a discontinuous Lagrangian family on the real line. 
\end{example}

We say that a Lagrangian family is {\bf smooth} if $L_M$ is a (smooth) subbundle. For example, given a two-form $\omega \in \Omega^2(M)$, a subbundle $E \subset TM$ or a bivector $\pi \in \X^2(M)$, the Lagrangian families
\[
\mathrm{Gr}(\omega) = \{u+\iota_u\omega \ | \ u \in TM\}, \qquad \mathrm{Gr}(E) = E \oplus E^{\circ}, \qquad \mathrm{Gr}(\pi) = \{\pi^{\sharp}(\xi)+\xi \ | \ \xi \in T^*M\}
\]are all smooth, and we refer to them as the {\bf graphs} of $\omega$, $E$ and $\pi$, respectively.

A {\bf Dirac structure} is a smooth Lagrangian family whose space of sections $\Gamma(L)$ is involutive with respect to the Dorfman bracket. Graphs corresponding to two-forms are Dirac structures iff the two-form is closed; similarly, the graph of a subbundle $E$ is a Dirac structure iff $E$ is the tangent bundle to a foliation, and the graph of a bivector is a Dirac structure iff the bivector is Poisson ---  equivalently, $L$ is the graph of a bivector exactly when $L$ meets the tangent bundle trivially:
\begin{align*}
    L \cap TM = 0.
\end{align*}

If $M$ is equipped with a Dirac structure $L_M$, and $f:N \to M$ is any smooth map, there is an induced {\bf pullback} Lagrangian family
\begin{align*}
 f^!(L_M):=\{u+f^*(\xi) \in \mathbb{T}N \ | \ f_*(u)+\xi \in L_M\}.
\end{align*}
Such a Lagrangian family may fail to be smooth:
\begin{example}
    For example, if $L_M$ is the Poisson structure $\pi_M = x_1\tfrac{\partial}{\partial x_1}\wedge \tfrac{\partial}{\partial x_2}$ on $M=\mathbb{R}^2$, and $i:\mathbb{R} \to M$ is $i(t) = (t,0)$, then $i^!(L_M)$ is the discontinuous Lagrangian family of Example \ref{ex : discontinuous Lagrangian family on R}.
\end{example}

However, if the pullback Lagrangian family $f^!(L_N)$ happens to be smooth, then it is automatically a Dirac structure \cite[Proposition 5.6]{Bu}, which we refer to as a {\bf pullback} Dirac structure.

In light of the discussion above, there is a canonical candidate to an \emph{induced} Dirac structure on a submanifold $X$ of a Dirac manifold $(M,L_M)$:\\

\begin{mdframed}
    A Dirac structure $L_M$ on $M$ {\bf induces} a Dirac structure $L_X$ on a submanifold $i:X \to M$ if the Lagrangian family $i^!(L_M)$ is smooth, in which case we call $i^!(L_M)$ the {\bf induced Dirac structure} on $X$.
\end{mdframed}

\begin{example}\label{ex : induced from foliation is foliation}\normalfont
 If $\mathscr{F}$ is a foliation on $M$, then $\mathrm{Gr}(\mathscr{F}) \colon=T\mathscr{F} \oplus N^*\mathscr{F}$ is a Dirac structure on $M$. In fact, a Dirac structure $L_M$ is of the form $\mathrm{Gr}(\mathscr{F})$ some foliation $\mathscr{F}$ on $M$ exactly when $\pr_{T}(L_M) = \pr_{T^*}(L_M)^{\circ}$ as subsets of $TM$ --- from which it follows by lower semicontinuity of $\mathrm{rank}\pr_{T}(L_M)$, and upper semicontinuity of $\mathrm{rank}\pr_{T^*}(L_M)^{\circ}$, that $\pr_{T}(L_M)$ is an (involutive) subbundle of $TM$. Moreover, if $i \colon X \to M$ is a submanifold, then
 \begin{align*}
  i^!\mathrm{Gr}(\mathscr{F})_x = i_*^{-1}(T_x\mathscr{F}) \oplus i^*(N^*_x\mathscr{F}) = i^*(N^*_x\mathscr{F})^{\circ} \oplus i^*(N^*_x\mathscr{F})
 \end{align*}for $x \in X$ shows that, if $i^!\mathrm{Gr}(\mathscr{F})$ is a Dirac structure on $X$, it must correspond to a foliation $\mathscr{F}_X$ on $X$. Note that $\mathscr{F}_X$ is nothing but the intersection of the ambient foliation $\mathscr{F}$ with $X$.
\end{example}

\subsection{Poisson-Dirac submanifolds}

Of special interest in this note is the case in which a Poisson structure induces, in the sense above, a Poisson structure on a submanifold:

\begin{definition}
 A {\bf Poisson-Dirac} submanifold $i:X \to M$ of a Poisson manifold
 $(M,\pi_M)$ is one in which $i^!\mathrm{Gr}(\pi_M)$ is smooth (and hence Dirac), and $i^!\mathrm{Gr}(\pi_M)$ meets $TX$ trivially. 
\end{definition}
That is: a Poisson-Dirac submanifold is one which inherits a Dirac structure which is Poisson.

\begin{example}\label{ex: vector space}\normalfont
 If $A$ is a vector space, any constant bivector $\pi_A \in \wedge^2A$ is Poisson. A linear subspace $B \subset A$
 inherits a (constant) Dirac structure $i^!\mathrm{Gr}(\pi_A) \subset B \oplus B^*$, and this is Poisson exactly when
 \begin{align}\label{eq : pointwise induced Poisson}
  \pi_A^{\sharp}(B^{\circ}) \cap B = 0.
 \end{align}In that case, for each $\xi \in B^*$, the element $\pi_B^{\sharp}(\xi)$ is given by 
 $\pi_A^{\sharp}(\widetilde{\xi})$, where $\widetilde{\xi} \in A^*$ is any element such that $\xi = \widetilde{\xi}|_B$
 and $\pi_A^{\sharp}(\widetilde{\xi}) \in B$.
 \end{example}
 
\begin{lemma}\label{lem : basic induced}
If $(X,\pi_X)$ is a Poisson-Dirac submanifold of $(M,\pi_M)$, and
\begin{align*}
    & \mathrm{S}_X(x) \subset X, & \mathrm{S}_M(x) \subset M,
\end{align*}denote, respectively, the symplectic leaves of $(X,\pi_X)$ and of $(M,\pi_M)$ passing through $x \in X$, then
 \begin{enumerate}[a)]
  \item $T_x\mathrm{S}_X(x) = T_xX \cap T_x\mathrm{S}_M(x)$;
  \item $\mathrm{S}_X(x)$ contains every Hamiltonian curve of $(M,\pi_M)$ which starts at $x$ and stays inside $X$.
 \end{enumerate}
\end{lemma}
\begin{proof}
 The tangent space $T_x\mathrm{S}_X(x)$ is the image of $\pi_X^{\sharp} \colon T^*_xX \to T_xX$, which, by definition of $\pi_X$, consists of those $\pi_M^{\sharp}(\xi)$ which are tangent to $X$ at $x$, and this proves a). This in turn implies that if $f \in C^{\infty}(I \times M)$ is a function whose Hamiltonian flow $\phi_{\mathrm{H}_f}^{t,0}$ is such that $\phi_{\mathrm{H}_f}^{t,0}(x) \in X$, then $f_X \colon =(\id,i)^*(f) \in C^{\infty}(I \times X)$ satisfies $\phi_{\mathrm{H}_{f_X}}^{t,0}(x) = \phi_{\mathrm{H}_f}^{t,0}(x)$, which proves b).
\end{proof}

There is an important subtlety concerning Poisson-Dirac submanifolds. Namely: if $X$ inherits a Poisson structure $\pi_X$ from $(M,\pi_M)$, there exist two induced partitions of $X$: one by the leaves $\mathrm{S}_X(x)$ of $\pi_X$, and the other by the subsets $X \cap \mathrm{S}_M(x)$. In general, however, the latter partition can be wildly misbehaved, as the next example illustrates.

\begin{example}\label{ex : intersections not manifolds}\normalfont \emph{An example of a Poisson-Dirac submanifold in which the partition $X \cap \mathrm{S}_M$ is not by smooth manifolds.} Let $i$ be the embedding
 \begin{align*}
 & i \colon \R^2 \rmap (\R^4,\pi=\tfrac{\partial}{\partial x_1} \wedge \tfrac{\partial}{\partial x_2}+x_3\tfrac{\partial}{\partial x_3} \wedge \tfrac{\partial}{\partial x_4}), & i(x,y)=(x,y,f(x,y)^2,f(x,y)^2),
\end{align*}where $f \in C^{\infty}(\R^2)$ is any smooth function. We claim that $i^!\mathrm{Gr}(\pi) = \mathrm{Gr}(\tfrac{\partial}{\partial x} \wedge \tfrac{\partial}{\partial y})$. Indeed, let $Z \subset \mathbb{R}^4$ denote the subset where $x_3=0$, and let $F \colon =f^{-1}(0) \subset \mathbb{R}^2$ its preimage under $i$. Observe that, in the complement of $Z$, the given Poisson structure is symplectic, corresponding to the closed two-form
\[
\omega = \dd x_2 \wedge \dd x_1 +\dd x_4 \wedge \tfrac{\dd x_3}{x_3} \in \Omega^2(\mathbb{R}^4 \diagdown Z).
\]
However, the pullback (as a Dirac structure) of the graph of a closed two-form coincides with the graph of the pullback form:
\[
i^!\mathrm{Gr}(\omega) = \mathrm{Gr}(i^*(\omega)) = \mathrm{Gr}(\dd x_2 \wedge \dd x_1)|_{\mathbb{R}^2 \diagdown F} = \mathrm{Gr}(\nu)|_{\mathbb{R}^2 \diagdown F},
\]where $\nu = \partial{x_1} \wedge \partial_{x_2} \in \X^2(\mathbb{R}^2)$.

On the other hand, for $(x,y) \in F$, we have that
\[
\mathrm{Gr}(\pi)_{i(x,y)} = \mathrm{Gr}(\nu)_{(x,y)} \times T^*_{(0,0)}\mathbb{R}^2
\]whereas
\begin{align*}
& i_{*,(x,y)} \colon T_{(x,y)}\mathbb{R}^2 \to T_{(x,y)}\mathbb{R}^2 \times T_{(0,0)}\mathbb{R}^2, & i_{*,(x,y)}(u) = (u,0),\\& i^*_{(x,y)} \colon T^*_{(x,y)}\mathbb{R}^2 \times T^*_{(0,0)}\mathbb{R}^2 \to T^*_{(x,y)}\mathbb{R}^2, & i^*_{(x,y)}(\xi,\eta) = \xi
\end{align*}reduce to the canonical injection and projection in the first factor. Therefore,
\begin{align*}
i^!\mathrm{Gr}(\pi)_{(x,y)} & = \{ u + i^*(\xi,\eta) \ | \ i_*(u)+(\xi,\eta) \in \mathrm{Gr}(\pi)_{i(x,y)}\}\\
& = \{ u + \xi \ | \ (u,0)+(\xi,\eta) \in \mathrm{Gr}(\pi)_{i(x,y)}\}\\
& = \{ u + \xi \ | \ (u,0)=\pi_{i(x,y)}(\xi,\eta)\}\\
& = \{ u + \xi \ | \ u=\nu_{i(x,y)}^{\sharp}(\xi)\}\\
& = \mathrm{Gr}(\nu)_{(x,y)}.
\end{align*} 

There are three important takeaways from this example:
\begin{enumerate}[a) ]
    \item Any closed subset $F \subset \R^2$ arises as the preimage of a suitable smooth function $f$ \cite[Section 2, Chapter 2]{Hirsch}. Therefore, the intersection of the embedding with a symplectic leaf can be very ill-behaved.
    \item Illustrating our previous point, if we take for $f$ the smooth function
    \begin{align*}
  f(x,y) = \begin{cases}
            e^{-\tfrac{1}{x^2}}\left(y-\sin(\tfrac{1}{x})\right), & x\neq 0;\\
            0 & x=0,
           \end{cases}
 \end{align*}then $F \subset \R^2$ is a typical example of a connected topological space which fails to be path-connected. (So, in particular, it is \emph{not} a manifold!)
    
    \noindent Nevertheless, that $X$ meets symplectic foliation in a rather pathological way does \emph{not} prevent the existence of a Poisson-Dirac structure on $X$.
    
    \item Hamiltonian flows of the Poisson-Dirac structure need \emph{not} lie in a single leaf of the ambient manifold. Indeed, any non-zero function with non-trivial zero locus is in fact a counterexample to the widely believed claim that the symplectic foliation of a Poisson-Dirac submanifold arises as the intersection of the ambient symplectic foliation with the submanifold. We call this a \underline{jumping phenomenon} to highlight the counterintuitive but important fact that the Hamiltonian flow
\[
 t \mapsto \phi^{t,0}_{\mathrm{H}_{f_X}}(x)
\]of a function $f_X \in C^{\infty}(X)$ may ``jump" between different leaves of the ambient manifold.
\end{enumerate}

In the concrete example b) above, something more striking occurs:
because $X=i(\mathbb{R}^2)$ is \emph{symplectic}, any two points are connected by a Hamiltonian path; however, because $F$ (which $i$ maps into a singular leaf) is connected but \emph{not} path connected, there exist $p_0,p_1 \in F$ which cannot possibly be endpoints of a Hamiltonian path in the ambient manifold $\mathbb{R}^4$ which lies in $X$.
\end{example}

Because a general Poisson-Dirac submanifold may intersect the ambient symplectic foliation in a poorly behaved fashion, the following language is in order: given an arbitrary subset $Y$ of a smooth manifold $M$, we shall say that two points $y_0,y_1$ in $Y$ lie in the same \emph{smooth path-connected component} if a smooth curve $c \colon I \to M$ exists, such that, for some $\epsilon>0$,
\begin{align*}
 c|_{[0,\epsilon]} = y_0, \qquad\qquad c(I) \subset Y, \qquad\qquad c|_{[1-\epsilon,1]} = y_1,
\end{align*}This is an equivalence relation on $Y$, and we denote by $\langle\langle Y \rangle\rangle_y$ the equivalence class of $y \in Y$. In this language, Lemma \ref{lem : basic induced} asserts that the partition of $X$ given by
\begin{align}\label{eq : ancillary partition}
 \mathrm{S}'_X(x) \colon =\langle\langle X \cap \mathrm{S}_M(x) \rangle\rangle_x
\end{align}\emph{refines} the partition $\mathrm{S}_X$ of $X$ given by the symplectic leaves of the induced Poisson structure, because any smooth path inside an ambient symplectic leaf is in fact a Hamiltonian path. Even when the partition $\mathrm{S}'_X$ consists exclusively of smooth manifolds, it may still be a strict refinement of $\mathrm{S}_X$:

\begin{example}\label{ex : jumping phenomenon}\normalfont
 The embedding
\begin{align*}
 & i \colon \R^2 \rmap (\R^4,\pi=\tfrac{\partial}{\partial x_1} \wedge \tfrac{\partial}{\partial x_2}+x_3\tfrac{\partial}{\partial x_3} \wedge \tfrac{\partial}{\partial x_4}), & i(x,y)=(x,y,x^3,0).
\end{align*}meets the leaves in smooth submanifolds:
\begin{align*}
 X \cap \mathrm{S}_M(ae_3+be_4) = \begin{cases}
                           \R_{\pm} \times \R, & \pm a >0;\\
                           \varnothing, & a = 0 \neq b;\\
                           \{0\} \times \R, & a=b=0.
                          \end{cases}
\end{align*}

The induced Lagrangian family $i^!\mathrm{Gr}(\pi_M)$ is the (symplectic) Poisson structure $\mathrm{Gr}(\tfrac{\partial}{\partial x}\wedge \tfrac{\partial}{\partial y})$. This can be argued in exactly the same fashion as in Example \ref{ex : intersections not manifolds}: on the open set $i^{-1}(\mathbb{R}^4 \diagdown Z)$ the family $i^!\mathrm{Gr}(\pi)$ is merely the graph of the pullback of the form which corresponds to $\pi$ on $\mathbb{R}^4 \diagdown Z$. Along points of $i^{-1}(Z)$, we again have that $i_*$ and $i^*$ are respectively the inclusion and the projection in the first factor, and so the computation along points of $Z$ is \emph{ipsis litteris} that of \ref{ex : intersections not manifolds}. 

Therefore, the intersection of the Poisson-Dirac submanifold $X$ with the symplectic foliation is a strict refinement of the partition of $X$ by its own symplectic foliation -- which consists of a single leaf.

\end{example}

\subsection{Clean Poisson-Dirac submanifolds}

In light of the discussion above, a first ``regularity'' condition to be imposed on Poisson-Dirac submanifolds is to require that there be no surprises when it comes to the induced singular foliation:

\begin{definition}\label{def : cleanly induced}
 A Poisson-Dirac submanifold $X$ of a Poisson manifold $(M,\pi_M)$ is {\bf clean} if the partition $\mathrm{S}_X$ by leaves of its induced Poisson structure $\pi_X$ coincides with the partition $\mathrm{S}'_X$ induced by the partition $\mathrm{S}_M$ of $M$ by leaves of $\pi_M$.
\end{definition}

The adjective ``clean'' is justified by the following

\begin{lemma}
 For a Poisson-Dirac submanifold $X$ of $(M,\pi_M)$, the following conditions are equivalent:
 \begin{enumerate}[i)]
  \item $X$ meets the leaves of $(M,\pi_M)$ cleanly --- that is, $X\cap \mathrm{S}_M(x)$ is an embedded submanifold of $\mathrm{S}_M(x)$ and $T(X\cap \mathrm{S}_M(x)) = TX \cap T\mathrm{S}_M(x)$;
  \item $X$ is a clean Poisson-Dirac submanifold.
 \end{enumerate}
\end{lemma}
\begin{proof}
Recall that $\mathrm{S}_X'$ refines $\mathrm{S}_X$, that is, each $\mathrm{S}_X(x)$ is a disjoint union
\begin{align*}
    & \mathrm{S}_X(x) =  \coprod_{y \in \Upsilon(x)}\mathrm{S}'_X(y), & \Upsilon(x) \subset \mathrm{S}_X(x).
\end{align*}

\noindent \emph{i) implies ii).} If $X$ meets $\mathrm{S}_M(x)$ cleanly for every $x \in X$, then $\mathrm{S}'_X(x)$ --- the smooth path connected component of $X \cap \mathrm{S}_M(x)$ through $x$ --- is by Lemma \ref{lem : app : leaf like clean is initial} an initial submanifold of $\mathrm{S}_X(x)$, with
\begin{align*}
    T_x\mathrm{S}'_X(x) & = T_xX \cap T_x\mathrm{S}_M(x) \\
    & = T_x\mathrm{S}_X(x).
\end{align*}Therefore, $\mathrm{S}'_X(x)$ is an open submanifold of $\mathrm{S}_X(x)$. Since the latter is connected and partitioned by $\mathrm{S}'_X$, we deduce that
\begin{align*}
    \mathrm{S}'_X(x) = \mathrm{S}_X(x)
\end{align*}for every $x \in X$. Therefore $X$ is a clean Poisson-Dirac submanifold.

\noindent \emph{ii) implies i).} By hypothesis, the smooth path connected component $\mathrm{S}'_X(x)$ of $X \cap \mathrm{S}_M(x)$ through $x$ is the symplectic leaf $\mathrm{S}_X(x)$ of $(X,\pi_X)$ through $x$. This implies that $X \cap \mathrm{S}_M(x)$ is a disjoint union of initial submanifolds $\mathrm{S}_X(x)$, for which according to a) in Lemma  \ref{lem : basic induced}  $T_y\mathrm{S}_X(x) = T_yX \cap T_y\mathrm{S}_M(x)$ for all $y \in \mathrm{S}_X(x)$. By Proposition \ref{prop:app-clean}, this implies that $X$ and $\mathrm{S}_M(x)$ meet cleanly.
\end{proof}

\begin{remark}\normalfont
The cleanness condition for Poisson-Dirac submanifolds is reminiscent of the transversality condition for Poisson transversals. In the case of the latter, we demand that a submanifold $X$ of a Poisson manifold $(M,\pi_M)$ meet the leaves of $M$ transversally,
\begin{align*}
    & T_xM = T_xX+T_x\mathrm{S}_M(x), & x \in X,
\end{align*}and that the intersections $X \cap \mathrm{S}_M(x)$ be symplectic submanifolds. It then \ul{follows} that the connected components of $X \cap \mathrm{S}_M(x)$ are the symplectic leaves of a smooth Poisson structure on $X$. 

The cleanness condition can be rightfully regarded as a relaxation of the transversality condition above. However, in contrast to Poisson transversals, a submanifold which meets the leaves of a Poisson structure cleanly and symplectically need \emph{not} be Poisson-Dirac, as the next two examples show.
\end{remark}

\begin{example}\label{ex : symplectic manifolds do not assemble into a Poisson structure}\normalfont
\emph{An example from \cite[Example 3]{CF} in which a submanifold which meets leaves cleanly and symplectically need not inherit an induced singular foliation.} On $\mathbb{C}^3$, equipped with complex coordinates $(z_1,z_2,z_3)$, we consider the Poisson structure of constant rank corresponding to the foliation given by $\dd z_2 = 0$ and $\dd z_3 - z_2\dd z_1 = 0$, and the pullback of the standard symplectic form $\tfrac{i}{2}\sum_1^3 \dd z_i \wedge \dd \overline{z}_i$ to leaves. Then the locus $X$ of $z_3=0$ meets leaves cleanly and symplectically:
\begin{align*}
 \mathrm{S}'_X(z_1,z_2,0) = \begin{cases}
                           \mathbb{C} \times \{ (0,0)\} & \text{if} \ z_2 = 0;\\
                            \{(z_1,z_2,0)\} & \text{if} \ z_2 \neq 0.
                          \end{cases}
\end{align*}This shows that the partition $\mathrm{S}'_X$ cannot even arise from a singular foliation (for it is not lower-semicontinuous).
\end{example}

\begin{example}\label{ex: induced foliation but not Poisson structure}\normalfont
\emph{An example in which the singular foliation induced on a submanifold which meets leaves cleanly and symplectically need not come from a Poisson structure.} Let $X \subset \R^{2}$ be the open unit disk and let $M \colon =X \times X \subset \R^{4}$ be endowed with the Poisson structure
\begin{align*}
\pi_M=(x_1^{2}+x^{2}_2+x_1)\tfrac{\partial}{\partial x_1}\wedge \tfrac{\partial}{\partial x_2}
+(x_3^{2}+x^{2}_4-x_3)\tfrac{\partial}{\partial x_3}\wedge
 	\tfrac{\partial}{\partial x_4}.
\end{align*}
This is the product of two dimensional Poisson structures which vanish in the circles of radius $\tfrac{1}{2}$ and center $\left(-\tfrac{1}{2},0\right)$ and $\left(\tfrac{1}{2},0\right)$, respectively.

The image of the origin under the diagonal embedding
\begin{align*}
 & i \colon X \rmap M, & i(t,s)=(t,s,t,s)
\end{align*}is the leaf of $\pi_M$ which consists of the origin alone, so the intersection of $X$ with $M$ is clean at that point. The image of a point $(t,s)$ which satisfies $t^2+s^2\pm t=0$  intersects transversely a two-dimensional ambient symplectic leaf. The complement of the union of the two circles $X_0=X\diagdown\{ t^2+s^2\pm t=0\}$  embeds in the four open symplectic leaves of $\pi_M$,
and the induced Lagrangian family is 
\begin{align*}
 & i^!\mathrm{Gr}(\pi_M)|_{X_0}=\mathrm{Gr}(\pi_{X_0}), &
 \pi_{X_0}=\tfrac{(t^{2}+s^{2})^{2}-t^{2}}{2t^{2}+2s^{2}}\tfrac{\partial}{\partial t}\wedge \tfrac{\partial}{\partial s}.
\end{align*}
Observe that the above formula extends to the intersection of the image of $X$ with the 2-dimensional leaves of $M$. 
Therefore $X$ intersects the symplectic leaves of $M$ cleanly in symplectic submanifolds and these submanifolds fit into
a  foliation of $X$. However, the leafwise symplectic forms do not come from a smooth Poisson tensor on $X$, because
\begin{align*}
 \underset{t \to 0}{\lim } \ \underset{s \to 0}{\lim } \ \mathrm{Gr}(\pi_{X_0})=\langle \tfrac{\partial}{\partial s}-2\dd t, \tfrac{\partial}{\partial t}+2\dd s\rangle \neq \langle  \dd t, \dd s\rangle = \underset{s \to 0}{\lim } \ \underset{t \to 0}{\lim } \ \mathrm{Gr}(\pi_{X_0})
\end{align*}shows that $\pi_{X_0}$ does not even extend to a continuous bivector on $X$.
\end{example}

\subsection{Split Poisson-Dirac submanifolds}

Among all Poisson structures $\pi_X$ on a Poisson-Dirac submanifold $X$ of a Poisson manifold $(M,\pi_M)$, the induced one is characterized by the property that Hamiltonian \emph{vectors} of $X$ are also Hamiltonian vectors for $M$: that is, for each $f \in C^{\infty}_c(X)$ and $x \in X$, there is an extension $\widetilde{f} \in C^{\infty}_c(M)$ of $f$, such that
\begin{align*}
 \mathrm{H}_f(x) = \mathrm{H}_{\widetilde{f}}(x).
\end{align*}This Poisson-Dirac submanifold is clean exactly when Hamiltonian \emph{curves} in $X$ are also Hamiltonian curves in $M$: that is, for each $f \in C^{\infty}_c(X \times I)$ and $x \in X$, there is an extension $\widetilde{f} \in C^{\infty}_c(M \times I)$ of $f$, such that
\begin{align*}
 \phi^{t,0}_{\mathrm{H}_f}(x) = \phi^{t,0}_{\mathrm{H}_{\widetilde{f}}}(x).
\end{align*}This leads to the next step in our hierarchy of good behavior, in which we require that every Hamiltonian \emph{flow} of $X$ be the restriction of a Hamiltonian flow of $M$, as ensured by:

\begin{definition}\label{def : poisson-dirac}
 A Poisson-Dirac submanifold $X$ of a Poisson manifold $(M,\pi_M)$ is {\bf split} if its Hamiltonian vector fields are the restriction of Hamiltonian vector fields of $M$.
\end{definition}

In contrast to the clean condition, there is a more convenient formulation of Definition \ref{def : poisson-dirac} involving a splitting condition: we say that an {\bf orthogonal splitting}  of a Poisson manifold $(M,\pi_M)$ along a submanifold $X$ is a splitting $TM|_X = TX \oplus E$ in which
\begin{align*}
 \pi_M|_X = \pi_X + \pi_E, \qquad \qquad \pi_X \in \Gamma(\wedge^2 TX), \qquad \qquad  \pi_E \in \Gamma(\wedge^2 E),
\end{align*}in which case it follows that $i^!\mathrm{Gr}(\pi_M) = \mathrm{Gr}(\pi_X)$, and so $\pi_X$ is the Poisson structure induced on $X$ by $(M,\pi_M)$.\footnote{These should not be confused with the Lie-Dirac submanifolds discussed in \cite[Section 8.3]{CF} and introduced in \cite{Xu} under the name of ``Dirac submanifolds".}

\begin{lemma}\label{lem : poisson-dirac, equivalent conditions}
The submanifolds along which a Poisson manifold has an orthogonal splitting are exactly the split Poisson-Dirac submanifolds.
\end{lemma}
\begin{proof}
Let $X$ be a submanifold of a Poisson manifold $(M,\pi_M)$. If $E \subset TM|_X$ is an orthogonal splitting along $X$, then $E \times I \subset T(M \times I)|_{X \times I}$ is an orthogonal splitting for $(\pi_M,0) \in \X^2(M \times I)$ along $X \times I$. Let $f \in C^{\infty}_c(X \times I)$ be a smooth function, and consider $\alpha \in \Gamma(T^*(M \times I)|_{X \times I})$ defined by
\begin{align*}
 & \alpha|_{T(X \times I)} = \dd f, & \alpha|_{E \times I} = 0.
\end{align*}Then $\alpha = \dd \widetilde{f}|_{X \times I}$ for some function $\widetilde{f} \in C^{\infty}_c(M \times I)$, and the Hamiltonian flow of $f$ is by construction the restriction to $X$ of the Hamiltonian flow of $\widetilde{f}$.

Conversely, suppose $X$ is a split Poisson-Dirac submanifold. Then an orthogonal splitting along $X$ may be constructed as $E \colon =\sigma(T^*X)^{\circ}$, where $\sigma \colon T^*X \to T^*M$ is any linear map satisfying
\begin{align*}
 & \sigma(\xi)|_X=\xi,  & \pi_X^{\sharp}(\xi)=\pi_M^{\sharp}\sigma(\xi)
\end{align*}for all $\xi \in T^*X$. Because these conditions are convex, it suffices to show that a such linear map exists in an open neighborhood $U \subset X$ of a point $x \in X$. But if $x_1,...,x_n \in C^{\infty}_c(X)$ define a coordinate chart on $U$, and $\widetilde{x}_1,...,\widetilde{x}_n \in C^{\infty}_c(M)$ are the extensions granted by the split Poisson-Dirac condition, the linear map
\begin{align*}
 & T^*U \rmap T^*M|_U, & \dd x_i \mapsto \dd \widetilde{x}_i
\end{align*}does the job.
\end{proof}

\begin{example}\label{ex : clean but not Poisson-Dirac}\normalfont
\emph{An example from \cite[Example 6]{CF} of a clean Poisson-Dirac submanifold which is not split.} Consider the embedding $i \colon \R^2 \to (M,\pi_M)$, where $M=\R^4$ and
 \begin{align*}
 & \pi_M = x_1^2\tfrac{\partial}{\partial x_1} \wedge \tfrac{\partial}{\partial x_2}+x_3\tfrac{\partial}{\partial x_3} \wedge \tfrac{\partial}{\partial x_4}, & i(t,s)=(t^2,0,t,s).
\end{align*}Then $i^!\mathrm{Gr}(\pi_M) = \mathrm{Gr}(t\tfrac{\partial}{\partial x_1} \wedge \tfrac{\partial}{\partial x_2})$ shows that $i$ maps leaves into leaves, and so its image is a clean Poisson-Dirac submanifold. However, it is not split, since the Hamiltonian vector field of $s \in C^{\infty}(\R^2)$ is not $i$-related to any Hamiltonian vector field of $M$:
\begin{align*}
 \mathrm{H}_s = -t\tfrac{\partial}{\partial t} \sim v = -2x_1\tfrac{\partial}{\partial x_1}-x_3\tfrac{\partial}{\partial x_3} = -2x_1\tfrac{\partial}{\partial x_1}+\mathrm{H}_{x_4}.
\end{align*}Here $v \in \X(M)$ is not Hamiltonian, because no Hamiltonian vector field on $(M,\pi_M)$ restricts to $x_1\tfrac{\partial}{\partial x_1}$ on $X$.
\end{example}

\begin{example}\normalfont
 Poisson submanifolds $X$ of a Poisson manifold $(M,\pi_M)$ are exactly those split Poisson-Dirac submanifolds for which any splitting $TM|_X=TX \oplus E$ along $X$ is orthogonal.
\end{example}

\begin{example}\normalfont
 Poisson transversals $X$ in a Poisson manifold $(M,\pi_M)$ are split Poisson-Dirac submanifolds with the unique orthogonal splitting $TM|_X=TX \oplus \pi_M^{\sharp}(N^*X)$ along $X$.
\end{example}

\subsection{Coregular submanifolds}
Among split Poisson-Dirac submanifolds, Poisson submanifolds and Poisson transversals are distinguished further by the property that  they are either saturated by, or transverse to ambient symplectic leaves. That is, the symplectic directions transverse to the submanifold fit into a subbundle (the trivial one and a full normal bundle, respectively). This suggests that we consider those Poisson-Dirac submanifolds $X$ of $(M,\pi_M)$, for which the image of the map
\begin{align*}
 & \mathrm{Q}_X \colon N^*X \rmap NX, & \mathrm{Q}_X(\xi) \colon =\pi_M^{\sharp}(\xi)+TX
\end{align*}is a vector bundle:

\begin{definition}\label{def: coregular} 
  A split Poisson-Dirac submanifold of a Poisson manifold $(M,\pi_M)$ is a {\bf coregular Poisson-Dirac} submanifold if $\mathrm{Q}_X \colon N^*X \rmap NX$ has locally constant rank.
\end{definition}

Thus, Poisson submanifolds are those Poisson-Dirac submanifolds for which $\mathrm{Q}_X$ vanishes, while Poisson transversals are those for which $\mathrm{Q}_X$ is surjective.

\begin{lemma}\label{lem : coregular is PD}For a submanifold $X$ of a Poisson manifold $(M,\pi_M)$, the following are equivalent:
\begin{enumerate}[i)]
    \item $X$ is a coregular Poisson-Dirac submanifold;
    \item $\pi_M^{\sharp}(N^*X) \subset TM|_X$ is a vector subbundle which meets $TX$ trivially\end{enumerate}
\end{lemma}
\begin{proof}
If $X$ is such that $\pi_M^{\sharp}(N^*X)\oplus TX$ is a vector subbundle of $TM|_X$, then any splitting $TM|_X = TX \oplus E$ in which $\pi_M^{\sharp}(N^*X) \subset E$ is an orthogonal splitting along $X$, and the rank of $Q_X$ is that of the vector bundle $\pi_M^{\sharp}(N^*X)$. Hence ii) implies i). Conversely, a Poisson-Dirac submanifold $X$ in a Poisson manifold $(M,\pi_M)$ in particular has an induced Poisson structure, whence $\pi_M^{\sharp}(N^*X) \cap TX = 0$. Hence $\pi_M^{\sharp}(N^*X)\oplus TX$ is a vector subbundle of $TM|_X$ if $\mathrm{Q}_X \colon N^*X \to NX$ has constant rank. So i) implies ii).
\end{proof}

\begin{example}\label{ex : point}\normalfont
 Every point in a Poisson manifold is coregular.
\end{example}

\begin{example}\label{ex : PD but not coregular}\normalfont \emph{An example of a split Poisson-Dirac submanifold which is not coregular.} Let $M=\mathfrak{so}(3)^*$ be dual to the Lie algebra of the group of
  oriented isometries of Euclidean 3-space, equipped with its canonical linear Poisson structure
 \begin{align*}
\pi_M=x_1\tfrac{\partial}{\partial x_2}\wedge \tfrac{\partial}{\partial x_3}+x_2\tfrac{\partial}{\partial x_3}\wedge \tfrac{\partial}{\partial x_1}+x_3\tfrac{\partial}{\partial x_1}\wedge \tfrac{\partial}{\partial x_2}
\end{align*}whose leaves are concentric spheres and the origin. Consider 
the embedding $i \colon \R \to M$ given by $i(t)=(t,0,0)$. Then $\pi_M$ has an orthogonal splitting along the image $X$ of $i$, namely
\begin{align*}
 & E=\langle \tfrac{\partial}{\partial x_2}, \tfrac{\partial}{\partial x_3}\rangle, & \pi_M^{\sharp}(E^{\circ}) = 0.
\end{align*}Hence $X$ is a split Poisson-Dirac submanifold. Note however that
\begin{align*}
 \pi_M^{\sharp}(N^*X) = \begin{cases}
                         E & \text{if} \ x_1 \neq 0;\\
                         0 & \text{if} \ x_1 = 0
                        \end{cases}
\end{align*}whence $X$ is not coregular.
\end{example}

The hierarchy of submanifolds illustrated in Figure \ref{fig:hierarchyintro1} is therefore strict, since there exist:

--- Poisson-Dirac submanifolds which are not clean (e.g. the jumping phenomenon of Example \ref{ex : jumping phenomenon});

--- clean Poisson-Dirac submanifolds which are not split (e.g. Example \ref{ex : clean but not Poisson-Dirac});

--- split Poisson-Dirac manifolds which are not coregular (e.g. Example \ref{ex : PD but not coregular}).

Nevertheless, for certain types of Poisson manifolds $(M,\pi_M)$, certain classes in the hierarchy illustrated in Figure \ref{fig:hierarchyintro1} may coincide. For instance:

\begin{example}\label{ex : coisotropic induced is PS}\normalfont
A submanifold $X$ in a Poisson manifold $(M,\pi_M)$ is called {\bf coisotropic} if
\[
\pi_M^{\sharp}(N^*X) \subset TX.
\]
A coisotropic which is Poisson-Dirac is necessarily a Poisson submanifold.
\end{example}

Another useful example occurs when the ambient manifold has locally constant rank:
\begin{proposition}\label{pro : regular holds no surprises} A Poisson-Dirac submanifold of a Poisson manifold of locally constant rank is necessarily coregular of locally constant rank.
\end{proposition}

\begin{proof}
If $(M,\pi_M)$ has locally constant rank, then
\begin{align*}
 \mathrm{Gr}(\pi_M) = \mathcal{R}_{\omega_M}\mathrm{Gr}(\mathscr{F}) \colon =\{u+\iota_u\omega_M+\xi \ | \ u+\xi \in \mathrm{Gr}(\mathscr{F})\},
\end{align*}where $T\mathscr{F}=\pi_M^{\sharp}(T^*M)$ and $\omega_M$ is a two-form on $M$ for which the endomorphism $\pi_M^{\sharp}\omega_M^{\sharp}$ restricts to the identity on $T\mathscr{F}$. Then
\begin{align*}
 i^!\mathrm{Gr}(\pi_M) = \mathcal{R}_{i^*(\omega_M)}i^!\mathrm{Gr}(\mathscr{F})
\end{align*}shows that $i^!\mathrm{Gr}(\pi_M)$ is Dirac exactly when $i^!\mathrm{Gr}(\mathscr{F})$ is Dirac, which according to Example \ref{ex : induced from foliation is foliation}, means that $i^!\mathrm{Gr}(\pi_M)$ must be a Poisson structure of locally constant rank if it is smooth. Moreover,
\[
 TX \cap T\mathscr{F}|_X \ \text{smooth} \quad \Longleftrightarrow \quad  TX + T\mathscr{F}|_X \ \text{smooth} \quad \Longleftrightarrow \quad N^*X \cap N^*\mathscr{F}\ \text{smooth};
\]because the leftmost space is $T\mathscr{F}_X$ and the rightmost is $\ker \mathrm{Q}_X$, $X$ is coregular.
\end{proof}

\section{Poisson submersions with Poisson fibres}\label{sec : coregular Poisson submersions}

We next turn to a ``compatibility'' condition for a surjective submersion between Poisson manifolds
\begin{align}\label{eq : submersion between Poisson manifolds}
 p \colon (\Sigma,\pi_{\Sigma}) \rmap (M,\pi_M).
\end{align}
There are two natural conditions to impose:
\begin{enumerate}[a)]
 \item that $p$ be a Poisson map;
 \item that fibres of $p$ be Poisson-Dirac submanifolds.
\end{enumerate}
Condition a) appears naturally in Poisson geometry:

\begin{example}\label{ex : symplectic groupoid}
Let $G \rightrightarrows M$ be a Lie groupoid. A two-form $\omega_G \in \Omega^2(G)$ is {\bf multiplicative} if
\[
\mathrm{m}^*(\omega_G)=\pr_1^*(\omega_G)+\pr_2^*(\omega_G),
\]where $\mathrm{m}$, $\mathrm{s}$ and $\mathrm{t}$ denote respectively the multiplication, source, and target maps. 

Then, if $\omega_G$ is symplectic, a unique Poisson structure $\pi_M$ exists on $M$, such that\begin{align*}
 & \mathbf{s} \colon (G,\omega_G) \to (M,\pi_M), & \mathbf{t} \colon (G,\omega_G) \to (M,-\pi_M)
\end{align*}are Poisson maps. 
\end{example}

Conversely, given a Poisson manifold $(M,\pi_M)$, there exists a submersion as in (\ref{eq : submersion between Poisson manifolds}), in which $\pi_{\Sigma}$ is symplectic and $p$ is Poisson --- this is what is called a \emph{symplectic realization} \cite{Ka,CDW}. A symplectic realization is said to be \emph{complete} if, for every for a function $f \in C^{\infty}(M)$, the Hamiltonian vector field on $\Sigma$ corresponding to $f \circ p$ is complete if that of $f$ is. Complete symplectic realizations exist exactly when $(M,\pi_M)$ arises from a symplectic groupoid $(G,\omega_G)$ as in Example \ref{ex : symplectic groupoid} \cite[Theorem 8]{CF}.
\\

Condition b), on the other hand, appears naturally in the context of this paper if one is to interpret a submersion as a family $p^{-1}(x)$ of submanifolds of $\Sigma$, parametrized by $M$. The simplest condition to consider in this direction is:

\begin{example}[Coupling Poisson structure]\label{ex : coupling Poisson structure}\normalfont
 A submersion $p \colon (\Sigma,\pi_{\Sigma}) \to M$ from a Poisson manifold is \emph{coupling} if its fibres are Poisson transversals.
\end{example}
For example: every Poisson manifold is coupling around an embedded symplectic leaf. Explicitly: if $(\mathrm{S}_M(x),\omega_{\mathrm{S}_M(x)})$ is a closed symplectic leaf of $(M,\pi_M)$, and $p \colon M \supset U \to \mathrm{S}_M(x)$ is a tubular neighborhood of $\mathrm{S}_M(x)$, then (shrinking $U$ if need be) $p \colon (U,\pi_M) \to \mathrm{S}_M(x)$ is coupling.\\

A simple example in which both conditions a) and b) are met is given by the following

\begin{example}[Vertical Poisson structures]\label{ex:vertical}\normalfont
 Any vertical Poisson structure $\pi_{\Sigma} \in \Gamma(\wedge^2V)$, $V=\ker p_*$, turns a surjective submersion
 $p \colon \Sigma \to M$ into a Poisson submersion $p \colon (\Sigma,\pi_{\Sigma}) \to (M,0)$ whose fibers are Poisson submanifolds.
\end{example}

The last two examples, with coregular Poisson-Dirac fibres, motivate the following:
\begin{definition}
 A submersion $p \colon (\Sigma,\pi_{\Sigma}) \to M$ has {\bf Poisson fibres} if its fibers
 are Poisson-Dirac submanifolds.
\end{definition}

\begin{example}\label{ex : SPF which is not coregular}
    Consider the submersion
    \begin{align*}
        & p: \Sigma = \mathbb{R}^3 \rmap \R^2 = M, & p(x_1,x_2,x_3)=(x_1,x_2).
    \end{align*}Then fibres of $p$ are Poisson-Dirac submanifolds of $\Sigma$ when equipped with the Poisson structure
    \[
    \pi_{\Sigma} = x_3\tfrac{\partial}{\partial x_1} \wedge \tfrac{\partial}{\partial x_2} \in \X^2(\Sigma),
    \]all fibres inheriting the zero Poisson structure. Note however that none of the fibres is a \emph{coregular} Poisson-Dirac submanifold.
\end{example}

It turns out that the fibres of a \ul{Poisson submersion} with Poisson fibres are automatically coregular. In fact,

\begin{theorem}\label{thm : coregular submersion pulls back coregular}
  A Poisson submersion $p \colon (\Sigma,\pi_{\Sigma}) \to (M,\pi_M)$ in which each tangent space $T_xp^{-1}(px) \subset (T_x\Sigma,\pi_{\Sigma,x})$ inherits a Poisson structure --- that is, $\pi_{\Sigma}^{\sharp}(V^{\circ}) \cap V = 0$ --- automatically has coregular Poisson-Dirac submanifolds for fibres. Moreover, for a coregular Poisson-Dirac submanifold $Y \subset M$,
  \begin{enumerate}[a)]
   \item $X \colon =p^{-1}(Y) \subset \Sigma$ is a coregular Poisson-Dirac submanifold;
   \item $p \colon (X,\pi_X) \to (Y,\pi_Y)$ is a Poisson submersion with Poisson fibres.
  \end{enumerate}
\end{theorem}
\begin{proof}
 For a coregular Poisson-Dirac submanifold $Y$ of $(M,\pi_M)$, we have that
 \begin{align*}
  TY \oplus \pi_M^{\sharp}(N^*Y) \subset TM|_Y
 \end{align*}is a vector subbundle. Taking the preimage under $p_*$,
 \begin{align*}
  TX + p_*^{-1}(\pi_M^{\sharp}(N^*Y)) \subset T\Sigma|_X
 \end{align*}is a vector subbundle. Now, because $p$ is Poisson,
 \begin{align*}
  TX + p_*^{-1}(\pi_M^{\sharp}(N^*Y)) = TX + \pi_{\Sigma}^{\sharp}(N^*X),
 \end{align*}and because the fibres of $p$ are Poisson-Dirac,
 \begin{align*}
  TX \cap \pi_{\Sigma}^{\sharp}(N^*X) = V \cap \pi_{\Sigma}^{\sharp}(p^*(\ker \mathrm{Q}_Y)) \subset V \cap \pi_{\Sigma}^{\sharp}(V^{\circ}) = 0.
 \end{align*}By Lemma \ref{lem : coregular is PD}, it follows that $X$ is a coregular Poisson-Dirac submanifold --- and specializing to the case where $Y$ is a point, we deduce that $p \colon (\Sigma,\pi_{\Sigma}) \to (M,\pi_M)$ is a Poisson submersion with coregular Poisson-Dirac submanifolds as fibres. Moreover, the diagram of Poisson manifolds
 \begin{align*}
  \xymatrix{
  (X,\pi_X) \ar[d]_{p|_X} \ar[r]^j & (\Sigma,\pi_{\Sigma}) \ar[d]^p\\
  (Y,\pi_Y) \ar[r]_i & (M,\pi_M)
  }
 \end{align*}satisfies the condition of \cite[Lemma 3]{PT1.5}, which implies that $p|_X \colon (X,\pi_X) \to (Y,\pi_Y)$ is a 
 Poisson submersion. Explicitly: because $i$ induces a Poisson structure, for any $u_Y+\xi_Y \in \mathrm{Gr}(\pi_Y)$ there is $\xi_M \in T^*M$ such that $\xi_Y=i^*(\xi_M)$ and $i_*(u_Y)+\xi_M \in \mathrm{Gr}(\pi_M)$; because $p \colon (\Sigma,\pi_{\Sigma}) \to (M,\pi_M)$ is Poisson, there exists $u_{\Sigma} \in T\Sigma$ such that $i_*(u_Y)=p_*(u_{\Sigma})$ and $u_{\Sigma}+p^*(\xi_M) \in \mathrm{Gr}(\pi_{\Sigma})$. Because $TX\simeq TY \times_{TM}T\Sigma$ and $pj=ip|_X$, it follows that $(u_Y,u_{\Sigma}) + (p|_X)^*(\xi_Y) \in \mathrm{Gr}(\pi_X)$, and this implies that $p|_X$ is a Poisson submersion. It has Poisson fibres because both
 \begin{align*}
  & V_y \subset (T_y\Sigma,\pi_{\Sigma,y}), & T_xX \subset (T_x\Sigma,\pi_{\Sigma,x})
 \end{align*}inherit Poisson structures $\pi_{V,y} \in \wedge^2V_y$ and $\pi_{X,x} \in \wedge^2T_xX$, for all $y \in \Sigma$ and $x \in X$, and therefore $V_x \subset (T_x\Sigma,\pi_{X,x})$ inherits the Poisson structure $\pi_{V,x}$.
\end{proof}

In the spirit of the comment leading up to Proposition \ref{pro : regular holds no surprises}, one could say that, for the class of submanifolds which arise as fibres of Poisson submersions, the hierarchy of Figure \ref{fig:hierarchyintro2} collapses. Other pertinent examples of this phenomenon can be found in Section \ref{sec : Examples}.

\begin{remark}\normalfont
 A Poisson map $p \colon (\Sigma,\pi_{\Sigma}) \to (M,\pi_M)$ is automatically transverse to any Poisson transversal $Y \subset M$, whose preimage $X \colon =p^{-1}(Y) \subset \Sigma$ is again a Poisson transversal, and $p|_X \colon X \to Y$ becomes a Poisson map for the induced Poisson structures \cite[Lemma 1]{PT1.5}. Theorem \ref{thm : coregular submersion pulls back coregular} may be regarded as a possible analogue to this previous statement in the context of coregular Poisson-Dirac submanifolds. 
\end{remark}
Let us point out in passing that --- in contrast to Poisson transversals --- Poisson maps need \emph{not} meet coregular Poisson-Dirac submanifolds cleanly, and even when they do (as in the hypotheses of Theorem \ref{thm : coregular submersion pulls back coregular}), their preimage need not be Poisson-Dirac, as the next two examples illustrate.

\begin{example}\label{ex : Poisson maps need not meet coregular submanifolds cleanly}\normalfont \emph{An example of a Poisson map which does not meet leaves cleanly.} Consider the surjective Poisson map
 \begin{align*}
  \psi \colon \left(\R^4,\tfrac{\partial}{\partial x_1} \wedge \tfrac{\partial}{\partial x_2} + \tfrac{\partial}{\partial x_3} \wedge \tfrac{\partial}{\partial x_4} \right) \rmap \left(\R^2,x_1\tfrac{\partial}{\partial x_1} \wedge \tfrac{\partial}{\partial x_2} \right),\\
  \psi(x_1,x_2,x_3,x_4)=(x_2,x_3x_4-x_1x_2)
 \end{align*}Then $\psi$ does not meet cleanly the leaf consisting of the origin alone, since
 \begin{align*}
  \psi^{-1}(0) = \{ (x_1,0,x_3,x_4) \ | \ x_3x_4=0\}
 \end{align*}is not a manifold.
\end{example}

\begin{example}\label{ex : Poisson submersion without coregular fibres}\normalfont \emph{An example of a Poisson submersion whose fibres are not Poisson-Dirac.} Let the cotangent bundle $T^*M$ of a smooth manifold be equipped with its canonical symplectic form $\omega_{\mathrm{can}} \in \Omega^2(T^*M)$,
 \begin{align*}
  & \omega_{\mathrm{can}} \colon =-\dd\lambda_{\mathrm{can}}, \quad \lambda_{\mathrm{can}} \in \Omega^1(T^*M), & \lambda_{\mathrm{can}}(v)_{\xi} \colon =\langle \xi,p_*(v)\rangle, \quad v \in T_{\xi}T^*M.
 \end{align*}If $\pi_{\mathrm{can}} \in \X^2(T^*M)$ denotes the Poisson structure corresponding to $\omega_{\mathrm{can}}$, the canonical projection
 \begin{align*}
  p \colon (T^*M,\pi_{\mathrm{can}}) \rmap (M,0)
 \end{align*}defines a Poisson submersion, none of whose fibres is Poisson-Dirac. Observe that this example shows that the hypothesis in Theorem \ref{thm : coregular submersion pulls back coregular} that tangent spaces inherit Poisson structures cannot be removed. 
\end{example}

\subsection{Couplings over leaves}

 As discussed in the previous section, Poisson submersions with Poisson fibres are modeled on both vertical Poisson structures and coupling Poisson submersions, in the same way as coregular Poisson-Dirac submanifolds were modeled on Poisson submanifolds and Poisson transversals. 
 
 In order to introduce our remaining \emph{dramatis person\ae}, we carry out a closer examination of the coupling condition to pave the way for the discussion to follow. \\

 A Poisson bivector $\pi_{\Sigma} \in \X^2(\Sigma)$ is \emph{coupling} for a surjective submersion $p \colon \Sigma \to M$ when the fibres of $p$ are Poisson transversals. This is equivalent to $\mathrm{Gr}(\pi_{\Sigma})$ being transverse to the Dirac structure $\mathrm{Gr}(V) \colon =V\oplus V^{\circ}$ corresponding to the foliation by fibres of $p$. This is in turn equivalent (see \cite{BF,Vo}) to the existence of
\begin{enumerate}[a)]
    \item An Ehresmann connection $H$;
 \item A bivector $\pi \in \Gamma(\wedge^2V)$;
 \item A form $\omega \in \Gamma(\wedge^2V^{\circ})$,
\end{enumerate}
such that
\begin{align*}
 \mathrm{Gr}(\pi_{\Sigma}) = \mathcal{R}_{\omega}(H) \oplus \mathcal{R}_{\pi}(H^{\circ}) \colon =\{u+\pi^{\sharp}(\xi)+\omega^{\sharp}(u)+\xi \ | \ u+\xi \in H\oplus H^{\circ}\},
\end{align*}in which case
 \begin{tasks}(2)
   \task $\pi$ is Poisson;
   \task $\mathscr{L}_{\mathrm{h}(u)}\pi=0$;
   \task $\mathrm{curv}(u,v) = \pi^{\sharp} \dd \omega(\mathrm{h}(v),\mathrm{h}(u))$,
   \task $\dd\omega(\mathrm{h}(u),\mathrm{h}(v),\mathrm{h}(w))=0$;
  \end{tasks}for all $u,v,w \in \X(M)$. Here $\mathrm{h}$ denotes the horizontal lift of $H$ and $\mathrm{curv}(u,v) \colon = \mathrm{h}([u,v])-[\mathrm{h}(u),\mathrm{h}(v)]$ denotes its curvature. Finally, if the submersion
  \begin{align*}
   p \colon (\Sigma,\pi_{\Sigma}) \rmap (M,\pi_M)
  \end{align*}has connected fibres, then it is a Poisson map for some Poisson structure $\pi_M$ on $M$ exactly when $\omega$ is closed, in which case $\pi_M$ is symplectic, $\omega$ is the pullback of the closed form on $M$ corresponding to $\pi_M$, and $H$ is involutive \cite[Corollary 1]{PT1.75}.
  
\begin{example}[Symplectic base]\label{ex : PD submersion with symplectic base}\normalfont
 For a Poisson submersion $p \colon (\Sigma,\pi_{\Sigma}) \to (M,\pi_M)$, we have that
 \begin{align*}
  \mathrm{Gr}(\pi_{\Sigma}) \cap \mathrm{Gr}(V) = \{\pi_{\Sigma}^{\sharp}(p^*\xi)+p^*\xi  \ | \ \xi \in \ker \pi_M^{\sharp}\}.
 \end{align*}Therefore a Poisson submersion is automatically coupling if its base $(M,\pi_M)$ is symplectic.
\end{example}

\begin{theorem}\label{thm : restriction of PD submersion over leaf}
 A Poisson submersion with Poisson fibres $p \colon (\Sigma,\pi_{\Sigma}) \to (M,\pi_M)$ pulls symplectic leaves of 
 $M$ back to Poisson submanifolds of $\Sigma$, and $p$ restricts to coupling Poisson submersions
 \begin{align*}
  p|_{\mathrm{S}_{\Sigma}(x)} \colon (\mathrm{S}_{\Sigma}(x),\omega_{\mathrm{S}_{\Sigma}(x)}) \rmap (\mathrm{S}_M(px),\omega_{\mathrm{S}_M(px)})
 \end{align*}for each $x \in \Sigma$.
\end{theorem}
\begin{proof}
The symplectic leaf $(\mathrm{S}_M(px),\omega_{\mathrm{S}_M(px)})$ of $(M,\pi_M)$ through $px \in M$
is in particular a coisotropic submanifold; hence so its preimage $p^{-1}(\mathrm{S}_M(px)) \subset \Sigma$ under 
the Poisson map $p$. However, because $\mathrm{S}_M(px)$ is a coregular Poisson-Dirac submanifold, 
it follows from Theorem \ref{thm : coregular submersion pulls back coregular} 
that $p^{-1}\mathrm{S}_M(px)$ is a coregular Poisson-Dirac submanifold, and that 
\begin{align*}
 p|_{p^{-1}\mathrm{S}_M(px)} \colon (p^{-1}\mathrm{S}_M(px),\pi_{p^{-1}\mathrm{S}_M(px)}) \rmap (\mathrm{S}_M(px),\omega_{\mathrm{S}_M(px)})
\end{align*}is a Poisson submersion with Poisson fibres. This implies by Example \ref{ex : coisotropic induced is PS} that $p^{-1}\mathrm{S}_M(px)$ is a Poisson submanifold, and by Example \ref{ex : PD submersion with symplectic base} that $p \colon (p^{-1}\mathrm{S}_M(px),\pi_{p^{-1}\mathrm{S}_M(px)}) \to (S_M,\pi_{S_M})$ is a coupling Poisson submersion. Hence if $\mathrm{S}_{\Sigma}(x)$ denotes the symplectic leaf of $\Sigma$ through $x$, then
\begin{align*}
 p|_{\mathrm{S}_{\Sigma}(x)} \colon (\mathrm{S}_{\Sigma}(x),\omega_{\mathrm{S}_{\Sigma}(x)}) \rmap (\mathrm{S}_M(px),\omega_{\mathrm{S}_M(px)})
\end{align*}is a Poisson map (because the inclusion of a symplectic leaf is Poisson), and because $\mathrm{S}_M(px)$ is symplectic, $p|_{\mathrm{S}_{\Sigma}(x)} \colon \mathrm{S}_{\Sigma}(x) \to \mathrm{S}_M(px)$ must be a submersion, and again by Example \ref{ex : PD submersion with symplectic base}, it is automatically coupling.
\end{proof}

Beware that such couplings over leaves need \emph{not} be surjective, as illustrated by the example below.

\begin{example}\label{ex : couplings over leaves need not be surjective}\normalfont \emph{An example of a Poisson submersion with Poisson fibres whose restriction over a leaf is not surjective.} Consider the surjective Poisson submersion
 \begin{align*}
  & p \colon (\Sigma,\pi_{\Sigma}) = \left(\R^3,\tfrac{\partial}{\partial x} \wedge \tfrac{\partial}{\partial y}+(1+z^2)\tfrac{\partial}{\partial z} \wedge \tfrac{\partial}{\partial y}\right) \rmap \left(\R^2,\tfrac{\partial}{\partial x} \wedge \tfrac{\partial}{\partial y}\right)=(M,\pi_M), & p(x,y,z)=(x,y).
 \end{align*}Note that \[\mathrm{S}_{\Sigma}(0) = \{(\arctan(z),y,z) \ | \ y,z \in \R\},\]and therefore the restriction $p|_{\mathrm{S}_{\Sigma}(0)} \colon \mathrm{S}_{\Sigma}(0) \to \mathrm{S}_{M}(0) = M$ is not surjective.
\end{example}

\subsection{Pencils}

By the description of a coupling Poisson submersion
\[
 p \colon (\Sigma,\pi_{\Sigma}) \rmap (M,\omega_M),
\]given above, the Poisson structure $\pi_{\Sigma}$ on the total space splits as a sum
\[
 \pi_{\Sigma}=\pi_V+\pi_H,
\]where $\pi_V$ is the vertical Poisson structure inducing the given Poisson structures on fibres, and $\pi_H$ is a regular Poisson structure obtained by pulling $\omega_M$ back to a flat Ehresmann connection on $p \colon \Sigma \to M$. This trivially implies that $\pi_V,\pi_H$ \ul{commute}.

This phenomenon holds true in the setting proposed in \cite{Va} to generalize the coupling condition:

\begin{example}[Almost-coupling]\label{ex : almost-coupling Poisson submersion}\normalfont
 An \emph{almost-coupling} Poisson submersion $p \colon (\Sigma,\pi_{\Sigma}) \to (M,\pi_M)$ is one, such that an Ehresmann connection $H \subset T\Sigma$ exists, for which $\pi_{\Sigma}(H^{\circ},V^{\circ})=0$. A such connection splits $\pi_{\Sigma}$ as a sum $\pi_{\Sigma}=\pi_V+\pi_H$, where $\pi_V$ and $\pi_H$ are the vertical- and horizontal bivectors defined by
 \begin{align*}
& \pi_V^{\sharp}(\xi)=
\begin{cases}
    0, & \xi \in V^{\circ};\\
    \pi_{\Sigma}^{\sharp}(\xi), & \xi \in H^{\circ}
\end{cases}, & \pi_H^{\sharp}(\xi)=
\begin{cases}
    0, & \xi \in H^{\circ};\\
    \pi_{\Sigma}^{\sharp}(\xi), & \xi \in V^{\circ}.
\end{cases}
\end{align*}Moreover, because $p$ is Poisson, it follows that the horizontal bivector $\pi_H$ is of the form $\pi_H=\mathrm{h}(\pi_M)$, where $\mathrm{h} \colon \X^{\bullet}(M) \to \X^{\bullet}(\Sigma)$ denotes the horizontal lift (of multivectors) associated with $H$. This implies that in the induced bigrading $\X^{p,q}(\Sigma) = \Gamma(\wedge^p V \otimes \wedge^q H)$, we have
\begin{align*}
 [\pi_V,\pi_V] \in \X^{3,0}(\Sigma), \qquad [\pi_V,\pi_H] \in \X^{2,1}(\Sigma), \qquad [\pi_H,\pi_H] \in \X^{1,2}(\Sigma)\oplus \X^{0,3}(\Sigma).
\end{align*}Hence $\pi_{\Sigma}=\pi_V+\pi_H$ Poisson implies that $\pi_V$ and $\pi_H$ are commuting Poisson structures,
\begin{align*}
 [\pi_V,\pi_V]=0, \qquad [\pi_V,\pi_H]=0, \qquad [\pi_H,\pi_H]=0.
\end{align*}
\end{example}

This motivates our next

\begin{definition}\label{def : orthogonal pencil}\normalfont
 An {\bf orthogonal pencil} is a Poisson submersion $p \colon (\Sigma,\pi_{\Sigma}) \to (M,\pi_M)$ in which $\pi_{\Sigma}$ splits into commuting Poisson structures $\pi_{\Sigma} = \pi_V+\pi_H$, 
 where $\pi_V \in \Gamma(\wedge^2V)$ is a vertical bivector, and $\pi_H^{\sharp}(T^*\Sigma) \cap V = 0$. 
\end{definition}

Note that an orthogonal pencil $p \colon (\Sigma,\pi_{\Sigma}) \to (M,\pi_M)$ is automatically a Poisson submersion with Poisson fibres, and that the splitting $\pi_{\Sigma} = \pi_V+\pi_H$ is unique. Moreover, an almost-coupling submersion $p \colon (\Sigma,\pi_{\Sigma}) \to (M,\pi_M)$ is tantamount to an orthogonal pencil for which $\pi_H$ is tangent to an Ehresmann connection for $p$.

    \begin{example}\label{ex: orthogonal pencil not almost coupling}\normalfont \emph{An example of an orthogonal pencil which is not almost-coupling.} Consider on $\Sigma \colon =\mathbb{C}^2$, with coordinates
    \begin{align*}
        & z_0 = x_0+iy_0 \in \mathbb{C}, & z_1 = x_1+iy_1 \in \mathbb{C},
    \end{align*}the Euler and rotational vector fields: 
      \begin{align*}
    & \mathscr{E}_i=x_i\tfrac{\bd}{\bd x_i} + y_i\tfrac{\bd}{\bd y_i}, & \mathscr{V}_i=x_i\tfrac{\bd}{\bd y_i} - y_i\tfrac{\bd}{\bd x_i}.
    \end{align*}
    Then there is a surjective Poisson submersion 
  \begin{align*}
     & p \colon (\Sigma,\pi_\Sigma)\to (M,\pi_M) & p(z_0,z_1) = z_0,
    \end{align*}where $M \colon = \mathbb{C}$ and
    \begin{align*}
     & \pi_{\Sigma} = |z_0|^2(\mathscr{E}_0 + \mathscr{E}_1)\wedge (\mathscr{V}_0+\mathscr{V}_1), & \pi_M=|z_0|^2\mathscr{E}_0 \wedge \mathscr{V}_0.
\end{align*}
    Note that
    \begin{align*}
    \pi_{\Sigma}^{\sharp}(T^*\Sigma) = \begin{cases}
                                        \langle \mathscr{E}_0 + \mathscr{E}_1,\mathscr{V}_0+\mathscr{V}_1\rangle & \text{if} \ z_0 \neq 0;\\
                                        0 & \text{if} \ z_0 = 0
                                       \end{cases}
    \end{align*}meets the vertical bundle of $p \colon\Sigma \to M$ trivially, and therefore $p \colon (\Sigma,\pi_{\Sigma}) \to (M,\pi_M)$ is a Poisson submersion with Poisson fibres in which fibres all have the trivial Poisson structure. This implies that this is in fact an orthogonal pencil. However, it is not almost-coupling: an Ehresmann connection for which it is almost-coupling would coincide with $H=\pi_{\Sigma}^{\sharp}(T^*\Sigma)$ on the locus $z_0 \neq 0$, but
    \begin{align*}
\underset{x_0 \to 0}{\lim } \ \underset{x_1 \to 0}{\lim } \ \underset{y_0 \to 0}{\lim } \ \underset{y_1 \to 0}{\lim }H=\langle \tfrac{\partial}{\partial x_0}, \tfrac{\partial}{\partial y_0}\rangle \neq \langle \tfrac{\partial}{\partial x_1}, \tfrac{\partial}{\partial y_1}\rangle = \underset{x_1 \to 0}{\lim } \ \underset{x_0 \to 0}{\lim } \ \underset{y_1 \to 0}{\lim } \ \underset{y_0 \to 0}{\lim }H
    \end{align*}shows that $H|_{z_0 \neq 0}$ cannot extend to a global Ehresmann connection.
    \end{example}

While general Poisson submersions with Poisson fibres need not admit orthogonal splittings, they have in some sense a canonical candidate for the job:
  
\begin{proposition}\label{pro : singular horizontal foliation} The equivalence relation on the total space of a Poisson submersion with Poisson fibres $p \colon (\Sigma,\pi_{\Sigma}) \to (M,\pi_M)$ in which $x_0,x_1 \in \Sigma$ are equivalent iff there is $f \in C^{\infty}_c(M \times I)$, such that
\begin{align*}
 \phi_{\mathrm{H}_{f \circ p}}^{1,0}(x_0) = x_1
\end{align*}defines a singular foliation $\mathcal{S}_H$, each of whose leaves carries a canonical symplectic form.
\end{proposition}

\begin{proof}
First observe that, because $p$ is a Poisson map,
\begin{align*}
 \pi_{\Sigma}^{\sharp} \circ p^* : \Omega^1(M) \rmap \X(\Sigma)
\end{align*}defines a Lie algebra map, when $\Omega^1(M)$ is given the Koszul bracket
\begin{align*}
 [\xi,\eta]^{\pi_M}:=\mathscr{L}_{\pi_{M}^{\sharp}(\xi)}\eta-\iota_{\pi_{M}^{\sharp}(\eta)}\dd\xi.
\end{align*}This implies that the pullback vector bundle $A:=p^*(T^*M)$ over $\Sigma$ carries a structure of Lie algebroid, with bracket and anchor determined by
\begin{align*}
 & [p^*(\xi),p^*(\eta)]_A = p^*[\xi,\eta]^{\pi_M}, & \rho_A(p^*(\xi)) = \pi_{\Sigma}^{\sharp} \circ p^*(\xi).
\end{align*}In the usual manner, $A$ determines on $\Sigma$ a singular foliation $\mathcal{S}_H$, whose tangent space is $\pi_{\Sigma}^{\sharp}(V^{\circ})$. Hence its leaves are the equivalence classes of the equivalence relation described in the statement, and they are submanifolds of the symplectic leaves of $\Sigma$: $\mathcal{S}_H(x) \subset \mathrm{S}_{\Sigma}(x)$. Therefore,
\begin{align*}
 \omega_{\mathcal{S}_H(x)}(\mathrm{H}_{f \circ p},\mathrm{H}_{g \circ p})=\{f,g\} \circ p,
\end{align*}is just the restriction of the symplectic form on $\mathrm{S}_{\Sigma}(x)$ to $\mathcal{S}_H(x)$.
\end{proof}

\begin{remark}\normalfont
 The following asymmetry is noteworthy: while the singular horizontal foliation $\mathcal{S}_H$ is defined for all Poisson submersions with Poisson fibres, the partition into leaves of the Poisson-Dirac structures on fibres need not in general define a singular foliation, as Example \ref{ex : coregular submersion with discontinuous Poisson structure on fibres} below illustrates.
\end{remark}

In contrast to almost-coupling Poisson submersions, for a general Poisson submersion with Poisson fibres $p \colon (\Sigma,\pi_{\Sigma}) \to (M,\pi_M)$, the singular foliation $\mathcal{S}_H$, with leaves canonically equipped with symplectic forms (described in Proposition \ref{pro : singular horizontal foliation}) need not in general arise from a Poisson structure --- just as, for a general Poisson submersion with Poisson fibres, the induced Poisson structures $\pi_{p^{-1}(x)}$ on the fibres $p^{-1}(x)\subset \Sigma$ need not vary smoothly with $x \in M$ --- that is, there need not be any vertical bivector $\pi_V \in \Gamma(\wedge^2V)$ which restricts on $p^{-1}(x)$ to $\pi_{p^{-1}(x)}$. In fact, these conditions are simultaneously satisfied:

\begin{theorem} \label{thm : orthogonal splittings}
For a Poisson submersion with Poisson fibres $p:(\Sigma,\pi_{\Sigma}) \to (M,\pi_M)$, the following assertions 
are equivalent:
\begin{enumerate}[i)]
\item It admits an orthogonal pencil.
\item There is a Poisson structure $\pi_H \in \X^2(\Sigma)$ whose symplectic leaves are those of $\mathcal{S}_H$.
 \item The Poisson structures on fibres assemble into a vertical Poisson structure $\pi_V \in \Gamma(\wedge^2V)$.
\item The linear family 
 \begin{align*}
  \mathrm{Gr}(\pi_{\Sigma}) \cap (V \oplus T^*\Sigma) + V^{\circ}\subset \mathbb{T}^*\Sigma
 \end{align*}
 is a vector bundle.
 \end{enumerate}
\end{theorem}
\begin{proof}
If $\pi_{\Sigma}$ splits into a pencil $\pi_{\Sigma}=\pi_V+\pi_H$, where $\pi_V$ is a vertical bivector, and $\pi_H^{\sharp}(T^*\Sigma) \cap V = 0$, then
\begin{align}\label{eq : horizontal bivector}
 & \pi_H^{\sharp}(\xi) = \pi_{\Sigma}^{\sharp}(\xi), & \xi \in V^{\circ},
\end{align}which implies that the tangent space at $x \in \Sigma$ to the leaf $\mathcal{S}_H(x)$ of the singular foliation $\mathcal{S}_H$ of Proposition \ref{pro : singular horizontal foliation} coincides with  the tangent space $\pi_H^{\sharp}(T^*_x\Sigma)$ to the leaf of $\pi_H$ through $x$, and the symplectic form on those spaces coincides as well, being the pullback of that on the ambient space. Therefore i) implies ii). 

On the other hand, if a Poisson structure $\pi_H$ on $\Sigma$ exists whose singular symplectic foliation coincides with that of $\mathcal{S}_H$, then $\pi_H$ is unique, and satisfies (\ref{eq : horizontal bivector}) --- which is to say that $\pi_V:=\pi_{\Sigma}-\pi_H \in \X^2(\Sigma)$ is a vertical bivector, which induces on the fibres of $p$ the same Poisson structure as $\pi_{\Sigma}$ does. Therefore ii) implies iii). 

Next observe that if a vertical bivector $\pi_V \in \Gamma(\wedge^2V)$ induces on the fibres of $p$ the same Poisson structure as $\pi_{\Sigma}$, then
\begin{align*}
 \mathrm{Gr}(\pi_V) = \mathrm{Gr}(\pi_{\Sigma}) \cap (V \oplus T^*\Sigma) + V^{\circ},
\end{align*}and therefore iii) implies iv). Finally, note that  the pullback of the Lagrangian family
\begin{align*}
 L_{\Sigma} := \mathrm{Gr}(\pi_{\Sigma}) \cap (V \oplus T^*\Sigma) + V^{\circ}
\end{align*} under the inclusion $i:X \to \Sigma$ of any fibre $X$ of $p$ coincides with that of $\mathrm{Gr}(\pi_{\Sigma})$:
\begin{align*}
 i^!(L_{\Sigma}) = i^!\mathrm{Gr}(\pi_{\Sigma}).
\end{align*}This implies that $L_{\Sigma} = \mathrm{Gr}(\pi_V)$ for a vertical Poisson bivector $\pi_V \in \Gamma(\wedge^2V)$. Because on an open, dense subset $U \subset \Sigma$ on which the rank of $\pi_H^{\sharp}:=\pi_{\Sigma}^{\sharp}-\pi_V^{\sharp}$ is locally constant, there is an Ehresmann connection $H \subset TU$ for $p$ to which $\pi_H$ is tangent. Hence $p:(U,\pi_{\Sigma}) \to (pU,\pi_M)$ is an almost-coupling Poisson submersion, and $H$ induces the splitting $\pi_{\Sigma}|_U = \pi_V|U + \pi_H|_U$. By Example \ref{ex : almost-coupling Poisson submersion}, $\pi_V|_U$ and $\pi_H|_U$ are commuting Poisson structures, and because $U$ is dense in $\Sigma$, this implies that $\pi_{\Sigma} = \pi_V + \pi_H$ is an orthogonal pencil.
\end{proof}

\begin{example}\normalfont\label{ex : coregular submersion with discontinuous Poisson structure on fibres}
 \emph{An example of a Poisson submersion with Poisson fibres which is not an orthogonal pencil.} The quotient map of the action of scalar multiplication on $\Sigma=\mathbb{C}^{2} \diagdown \{0\}$,
\begin{align*}
 & \mathbb{C}^{\times} \times \Sigma \rmap \Sigma, & w \cdot (z_0,z_1):=(wz_0,wz_1),
\end{align*}gives rise to a submersion
\begin{align*}
 & p:\Sigma \to M=\mathbb{C}P^1, & p(z_0,z_1):=[z_0:z_1],
\end{align*}whose vertical bundle is spanned by $\mathscr{E}_0+\mathscr{E}_1$ and $\mathscr{V}_0+\mathscr{V}_1$ (using the notation of Example \ref{ex: orthogonal pencil not almost coupling}). The Poisson bivector
\begin{align*}
 \pi_{\Sigma} = \tfrac{1}{4}(\mathscr{E}_0 -\mathscr{E}_1)\wedge (\mathscr{V}_0-\mathscr{V}_1)
\end{align*}is $\mathbb{C}^{\times}$-invariant, and thus gives rise to a Poisson submersion $p:(\Sigma,\pi_{\Sigma}) \to (M,\pi_M)$, where $\pi_M=\mathscr{E}\wedge \mathscr{V}$, and where $\mathscr{E},\mathscr{V} \in \X(\mathbb{C}P^1)$ are vector fields which restrict in a standard affine chart restrict to the Euler and rotational vector fields (or their opposites) of $\mathbb{C}$. Note that the image of $\pi_{\Sigma}$ is spanned by $\mathscr{E}_0 -\mathscr{E}_1$ and $\mathscr{V}_0-\mathscr{V}_1$; hence the induced Poisson structure on the fibre $\Sigma_{[z_0:z_1]}$ over $[z_0:z_1] \in M$ is trivial if $z_0z_1 \neq 0$, and it is symplectic otherwise. Hence the Poisson structures on fibres do not even vary continuously, and by Theorem \ref{thm : orthogonal splittings}, no orthogonal pencil exists, nor is $\mathcal{S}_H$ the partition into leaves of a Poisson structure on $\Sigma$.
\end{example}

\section{Examples from Lie theory}\label{sec : Examples}

\subsection{Poisson-Lie groups}

A splitting of a Lie algebra $\mathfrak{d}$ into Lie subalgebras $\mathfrak{d} = \mathfrak{g} \oplus \mathfrak{h}$ is called a {\bf Manin triple} if $\mathfrak{d}$ is equipped with an invariant, symmetric bilinear pairing $\abk : S^2\mathfrak{d} \to \R$ for which the subalgebras $\mathfrak{g}$ and $\mathfrak{h}$ are Lagrangian. 

A {\bf $G$-invariant Manin triple} $(\mathfrak{d},\abk,\mathfrak{g},\mathfrak{h})$ is a Manin triple equipped
with a choice of Lie group $G$ with Lie algebra $\mathfrak{g}$,
and an extension $\mathrm{Ad}:G \acts \mathfrak{d}$ of the adjoint action of $G$ on $\mathfrak{g}$ which integrates
the Lie bracket action $\mathrm{ad}: \mathfrak{g}\acts \mathfrak{d}$ and such that 
\begin{align*}
 & \mathrm{Ad}_g[v,w] = [\mathrm{Ad}_g(v),\mathrm{Ad}_g(w)], & \langle v,w \rangle = \langle \mathrm{Ad}_g(v),\mathrm{Ad}_g(w) \rangle
\end{align*}for all $g \in G$ and $v,w \in \mathfrak{d}$. In
that case, the quotient representation $\mathrm{Ad}:G \acts \mathfrak{d}/\mathfrak{g}=\mathfrak{g}^*$ is
the coadjoint action. 

To a $G$-invariant Manin triple there corresponds a {\bf Poisson-Lie group} structure $\pi_G$ on $G$ \cite{Dr} (see also \cite[Section 5]{MeD}, whose perspective we espouse here) --- that is, a Poisson bivector $\pi_G \in \X^2(G)$, for which multiplication
\begin{align*}
 \mathrm{m} : (G,\pi_G) \times (G,\pi_G) \rmap (G,\pi_G)
\end{align*}is a Poisson map. Explicitly, this means that, for all $g_1,g_2 \in G$,
\begin{align*}
 \pi_{G,g_1g_2} = {l_{g_1}}_*\pi_{G,g_2} + {r_{g_2}}_*\pi_{G,g_1},
\end{align*}where $l_{g_1}(g)=g_1g$ and $r_{g_2}(g)=gg_2$ stand for left- and right-multiplication. Indeed, the $G$-invariant Manin triple  $(\mathfrak{d},\abk,\mathfrak{g},\mathfrak{h})$ defines an infinitesimal {\bf dressing action} $\varrho : \mathfrak{d} \to \Gamma(TG)$, uniquely determined by the condition that
\begin{align}\label{eq : dressing}
 & \mathrm{Ad}_{g}(\iota_{\varrho(v)}\theta^L_g) = \mathrm{pr}_{\mathfrak{g}}\mathrm{Ad}_g(v), & (g,v) \in G \times \mathfrak{d},
\end{align}where $\theta^L \in \Omega^1(G;\mathfrak{g})$ denotes the left-invariant Maurer-Cartan form of $G$.
The infinitesimal dressing action extends in fact to a linear map
\begin{align}\label{eq : epsilon}
 & \epsilon : \mathfrak{d} \rmap \Gamma(\mathbb{T}G), & \epsilon(v) = \varrho(v)+\langle \theta^L,v\rangle,
\end{align}which satisfies
\begin{align*}
 & \langle v,w\rangle = \langle \epsilon(v),\epsilon(w)\rangle,& [v,w] = [\epsilon(v),\epsilon(w)],
\end{align*}and for the ensuing isomorphism $\epsilon : G \times \mathfrak{d} \to \mathbb{T}G$, we have
\begin{align*}
 & TG = \epsilon(G \times \mathfrak{g}), & \mathrm{Gr}(\pi_G) = \epsilon(G \times \mathfrak{h}).
\end{align*}Note that, by definition of $\pi_G$, Poisson submanifolds of $(G,\pi_G)$ are unions of orbits of the infinitesimal action $\rho : \mathfrak{h} \to \X(G)$; that is, $\mathfrak{h}$-invariant submanifolds of $G$.

\subsection{Coregular Poisson-Dirac submanifolds from orbits}

Let $(G,\pi_G)$ be a Poisson-Lie group. An action $\alpha:G \times M \to M$ of $G$ on a Poisson manifold $(M,\pi_M)$ is \emph{Poisson} if
\begin{align*}
 \alpha : (G,\pi_G) \times (M,\pi_M) \rmap (M,\pi_M)
\end{align*}is a Poisson map.

\begin{remark}\normalfont
 A Poisson action of $(G,\pi_G)$ on $(M,\pi_M)$ need not act by Poisson diffeomorphisms of $(M,\pi_M)$, unless $G$ is equipped with the trivial Poisson structure $\pi_G=0$.
\end{remark}

\begin{lemma}[Orbits]\label{lem:orbits}\normalfont
 For an orbit $X$ of a Poisson action of a Poisson Lie group $(G,\pi_G)$ on a Poisson manifold $(M,\pi)$, the following assertions are equivalent:
\begin{enumerate}[i)]
  \item it is a coregular Poisson-Dirac submanifold;
  \item it is Poisson-Dirac;
  \item $T_xX \subset (T_xM,\pi_{M,x})$ is Poisson-Dirac for some $x \in X$.
\end{enumerate}
\end{lemma}
\begin{proof}
We need only show that condition iii) implies condition i). Let $g \in G$ and let $\xi \in N^*_{gx}X$. Because $\alpha^{-1}(X)=G \times X$, we have that
\begin{align*}
 \alpha^*(\xi) = (\alpha_x^*(\xi),\alpha_g^*(\xi)) = (0,\alpha_g^*(\xi)) \in N^*_{(g,x)}(G \times X) \subset  T^*_gG \times T^*_xM.
\end{align*}Because $\alpha$ is a Poisson map,
\begin{align*}
 \pi_{M,gx}^{\sharp}(\xi) & = \alpha_*(\pi_G,\pi_M)_{(g,x)}^{\sharp}\alpha^*(\xi) = \alpha_*(0,\pi_{M,x}^{\sharp}\alpha_g^*(\xi)) \\ & = {\alpha_g}_*\pi_{M,x}^{\sharp}\alpha_g^*(\xi) = 0,
\end{align*}In particular, this implies that
\begin{align*}
 & \pi_{M,gx}^{\sharp}(N^*_{gx}X) = {\alpha_g}_*\pi_{M,x}^{\sharp}(N^*_xX), & \pi_{M,gx}^{\sharp}(N^*_{gx}X) \cap T_{gx}X = {\alpha_g}_*\left(\pi_{M,x}^{\sharp}(N^*_xX)\cap T_{x}X \right),
\end{align*}where in the second equality we used condition iii). Hence $X$ is a coregular Poisson-Dirac submanifold.
\end{proof}

\begin{example}\label{ex : abelian actions}\normalfont
 If a vector space $A$ acts on a smooth manifold $M$, with action $\alpha:A \times M \to M$, and induced infinitesimal action
\begin{align*}
 & \rho : A \rmap \X(M), & \rho(v)_x:=\tfrac{d}{dt}\alpha(e^{-t}v,x)|_{t=0},
\end{align*}the induced map $\wedge^2\rho : \wedge^2\mathfrak{g} \to \X^2(M)$ maps into $A$-invariant Poisson 
bivectors on $M$, and
\begin{align*}
 & \alpha : (A,0) \times (M,\pi_M) \rmap (M,\pi_M), & \pi_M:=\wedge^2\rho(\pi_A)
\end{align*}is a Poisson action for any $\pi_A \in \wedge^2A$. 
Because the tangent space at $x \in M$ to the orbit $A \cdot x \subset M$ contains by 
construction $\pi_M^ {\sharp}(T^*_xM)$, they are all Poisson submanifolds. 
\end{example}

Note in the setting of Example \ref{ex : abelian actions} that, for any subspace $B \subset A$, the restricted action
\begin{align*}
 \alpha : (B,0) \times (M,\pi_M) \rmap (M,\pi_M)
\end{align*}is Poisson, but it need not be the case that its orbits $B \cdot x$ are 
coregular Poisson-Dirac submanifolds. For example, if $A=\C$ acts by translations on $M=\C$ with its standard symplectic structure, orbits of the subgroup $B=\R$ are not Poisson-Dirac. 

Let us borrow from Symplectic Geometry a useful setting in which orbits are automatically coregular Poisson-Dirac submanifolds:

\begin{definition}\label{def : positive bivector} 
 Let a vector space $A$ be equipped with a complex structure $J \in \mathrm{End}(A)$. A bivector $\pi_A \in \wedge^2A$ 
 is \emph{positive} if $J$ makes the symplectic leaves of $\pi_A$ into K\"ahler manifolds. 
\end{definition}Explicitly, $\pi_A$ is positive if $S:=\pi_A^{\sharp}(A^*)$ is a complex subspace,
$J:(S,\omega_S) \to (S,\omega_S)$ is a symplectic automorphism, and
 \begin{align*}
  & g_S:S \times S \rmap \R, & g_S(u,v):=\omega_S(u,Jv)
 \end{align*}is symmetric and positive-definite. Note that if $\pi_A$ is positive and $B \subset A$ is $J$-invariant,
 then for $\xi \in B^{\circ}$ such that $\pi_A^{\sharp}(\xi) \in B$,
 \begin{align*}
  0 = \langle \xi,J\pi_A^{\sharp}(\xi)\rangle = \pi_A(\xi,J^*\xi) = -\omega_S(\pi_A^{\sharp}(\xi),\pi_A^{\sharp}(J^*\xi))
  = \omega_S(\pi_A^{\sharp}(\xi),J\pi_A^{\sharp}(\xi)) = g_S(\pi_A^{\sharp}(\xi),\pi_A^{\sharp}(\xi))
 \end{align*}implies that the pertinent set $\pi_A^{\sharp}(B^{\circ}) \cap B$ in (\ref{eq : pointwise induced Poisson})
 is trivial, and so $B$ is a coregular Poisson-Dirac submanifold in $(A,\pi_A)$.

\begin{lemma}\label{lem : holomorphic actions}
 Let $A$ be a complex vector space equipped with a positive bivector $\pi_A \in \wedge^2A$. An action of $A$
 on a complex manifold $M$ by holomorphic transformations induces an $A$-invariant Poisson
 structure $\pi_M$ on $M$, with the property that, for all complex subspace $B \subset A$, the induced action 
 $(B,0) \acts (M,\pi_M)$ is Poisson and has coregular Poisson-Dirac submanifolds as orbits.
\end{lemma}
\begin{proof}
By Lemma \ref{lem:orbits}, $B \cdot x \subset M$ is a coregular Poisson-Dirac submanifold iff
\begin{align*}
 T_x(B \cdot x) \subset (T_xM,\pi_{M,x})
\end{align*}has an induced Poisson structure, and this happens exactly when
\begin{align*}
 T_x(B \cdot x) \subset (T_x(A \cdot x),\pi_{A \cdot x,x})
\end{align*}has an induced Poisson structure, where $\pi_{A \cdot x}$ is the Poisson structure on the Poisson submanifold $A \cdot x \subset (M,\pi_M)$ (as in Example \ref{ex : abelian actions}). Because $T_x(B \cdot x) \subset T_x(A \cdot x)$ is a complex subspace,
it suffices to check that $\pi_{A \cdot x,x}$ is positive --- 
and that is the case because the complex-linear infinitesimal action at $x$,
\begin{align*}
 \rho_x : (A,\pi_A) \rmap (T_x(A \cdot x),\pi_{A \cdot x,x})
\end{align*}induces an identification
\begin{align*}
 & \rho_x : (A/A_x,\pi_{A/A_x}) \diffto (T_x(A \cdot x), & A_x:=\ker(\rho_x),
\end{align*}where $\pi_{A/A_x}$ is in turn identified with the restriction of $\pi_A$ to 
$\wedge^2A_x^ {\circ}$ --- and is therefore positive.
\end{proof}

\begin{example}
 On $A=\C$, the bivector $\pi_A=\tfrac{\partial}{\partial x}\wedge \tfrac{\partial}{\partial y}$ is positive.
 Acting on $M=\C$ by translations, there ensues $\pi_M=\tfrac{\partial}{\partial x}\wedge \tfrac{\partial}{\partial y}$,
 while the action $w \cdot z = e^wz$ produces $\pi_M=(x^2+y^2)\tfrac{\partial}{\partial x}\wedge \tfrac{\partial}{\partial y}$.
\end{example}

\subsection{Poisson submersions with Poisson fibres from toric varieties}
The previous results can be used to produce Poisson submersions with Poisson fibres associated with the quotient
presentation of a (smooth) projective toric variety. By a \emph{toric variety} we mean 
a  variety with a Zariski open subset identified to an algebraic torus whose action on itself by multiplication extends to
an action on the variety.  Toric varieties are given by the combinatorial data encoded in 
a fan. Projective toric varieties can also be described 
via suitable polytopes, a viewpoint which is useful to highlight the symplectic geometry 
nature of toric varieties.

\begin{example}\normalfont
We recall Delzant's Hamiltonian quotient construction of symplectic toric manifolds and 
the necessary modifications to induce the projective toric
variety structure via a GIT (for Geometric Invariant Theory) quotient. A symplectic manifold $(M^{2n},\omega_M)$ is \emph{toric}
  if it comes equipped with an effective Hamiltonian action
 \begin{align*}
  T \acts (M,\omega_M) \stackrel{\mu_M}{\rmap} \mathfrak{t}^*,
 \end{align*}
where $T$ is an $n$-dimensional (compact) torus $T \simeq \mathbb{T}^n$
and  $\mathfrak{t}$ denotes its Lie algebra. A \emph{polytope} $\varDelta \subset \mathfrak{t}^*$ is a compact subset of the form
\begin{align}\label{eq:Delzant}
 \varDelta = \bigcap_{i=1}^d\{ \xi \in  \mathfrak{t}^* \ | \ \langle \xi,u_i\rangle \geqslant c_i\}
\end{align}where $c_i \in \R$ and $u_1,...,u_d \in  \mathfrak{t}$, which are thought of as normal
to the faces of $\varDelta$. Such a polytope is called \emph{Delzant} if $u_1,...,u_d$ can be rescaled to lie in the lattice 
$\varLambda \subset \mathfrak{t}$ which is the kernel of the exponential map $\exp : \mathfrak{t} \to T$, and 
at each vertex of $\varDelta$ 
the vectors normal to faces of $\varDelta$ through the vertex form a $\mathbb{Z}$-basis of
$\varLambda$.

To a Delzant polytope  $\varDelta\subset \mathfrak{t}^{*}$ as in (\ref{eq:Delzant}) one associates the exact sequence of tori:
\[
 1 \rmap N \stackrel{i}{\rmap} \mathbb{T}^d \stackrel{p}{\rmap} T \rmap 1
\]where $p$ is uniquely determined by the condition that $p_*(e_i)=u_i$, where $e_1,...,e_d$ stands for the standard basis of $\mathbb{R}^d=\mathrm{Lie}(\mathbb{T}^d)$. If $\omega_{\mathrm{std}} \in \Omega^2(\C^d)$ denotes the standard symplectic structure $\omega_{\mathrm{std}}=\tfrac{i}{2}\sum_{i=1}^d\dd z_i \wedge \dd \overline{z}_i$, the standard action of $\mathbb{T}^d$ on $\C^d$ by multiplication gives rise to a Hamiltonian action
\begin{align*}
 & \mathbb{T}^d \acts (\C^d,\omega_{\mathrm{std}}) \stackrel{\mu}{\rmap} \mathrm{Lie}(\mathbb{T}^d)^*, & \mu(z) = \sum_{i=1}^d\left( \tfrac{|z_i|^2}{2}+c_i \right)e_i.
\end{align*}Then $\mu_N:=i^*\mu: \C^d \to \mathfrak{n}^*$ is 
a moment map for the action of the subtorus $N \acts (\C^d,\omega_{\mathrm{std}})$. The map $\mu_N$ is proper,  
zero is a regular value, and the action of $N$ on $\mu_N^{-1}(0)$ is free.
Consequently, $M_{\varDelta}:=\mu_N^{-1}(0)/N$ is a compact smooth manifold endowed 
 with a residual action of $T=\mathbb{T}^d/N$. This action is Hamiltonian for the symplectic form
 on $M$ coming from Hamiltonian reduction of $\omega_{\mathrm{std}}$ and the image of the moment map is $\Delta$.

To obtain the algebro-geometric quotient construction  $\mu_N^{-1}(0)$ is
enlarged to an open dense subset $\Sigma_\Delta\subset \C^{d}$
which can be described in several equivalent ways: It is the saturation of $\mu_N^{-1}(0)$ by the action of the complexification 
$N_\C$ of $N$. It is the subset of $\C^{d}$ where $N_\C$ acts freely and with closed orbits. It is the collection 
of orbits the standard action $(\C^{\times})^d \acts \C^d$ which intersect $\mu_N^{-1}(0)$. More explicitly, the orbits of 
$(\C^{\times})^d \acts \C^d$
 are parametrized by subsets $I$ of $\{1,2,...,d\}$:
  \begin{align*}
 \C^d_I := \{ z \in \C^d \ | \ z_i = 0 \ \Leftrightarrow \ i \in I\}.
\end{align*} 
Each subset $F(\varDelta)_I \subset \varDelta$,
\begin{align*}
 F(\varDelta)_I := \{ \xi \in \varDelta \ | \ \langle \xi,u_i\rangle = c_i \ \Leftrightarrow \ i \in I\},
\end{align*}is a face of $\varDelta$ if nonempty and 
\begin{align*}
 & \Sigma_{\varDelta}=\bigcup_{F(\varDelta)_I \neq \varnothing} \C^d_I
\end{align*} 
These equivalent descriptions give a canonical identification of the compact quotient 
 $M_\Delta$ with the GIT quotient $\Sigma_{\varDelta}/N_{\C}$. The outcome
 is a complex (projective) structure on $M_\Delta$ together with a complex action of $(\C^{\times})^d/N_{\C}$ 
 with an open dense orbit.
 \end{example}

\begin{proposition}\label{pro : positive GIT}
Let $p:\Sigma_{\varDelta} \to M_{\varDelta}$ be the GIT quotient construction of the toric variety $M_{\varDelta}$. 
Then every positive bivector $\pi \in \wedge^2 \C^d$ induces Poisson structures 
$\pi_{\Sigma}$ on $\Sigma_{\varDelta}$ and $\pi_{M}$ on $M_{\varDelta}$, for which the quotient map $p:(\Sigma_{\varDelta},\pi_{\Sigma}) \to (M_{\varDelta},\pi_M)$ is a Poisson submersion with Poisson fibres.
\end{proposition}
\begin{proof}
 The complex vector space $A:=\C^d$ acts on $\C^d$ by the holomorphic transformation
 $(w_1,...,w_d) \cdot (z_1,...,z_d):=(e^{w_1}z_1,...,e^{w_d}z_d)$. By Lemma \ref{lem : holomorphic actions}, 
 there is an induced $A$-invariant Poisson structure $\Pi$ on $\C^d$, with the property 
 that $(\mathfrak{n}_{\C},0) \acts (\C^d,\Pi)$ is a Poisson action with coregular Poisson-Dirac submanifolds as orbits, 
 where $\mathfrak{n}_{\C}=\mathrm{Lie}(N_{\C})$. Because $\Sigma_{\varDelta} \subset \C^d$ is a union of $A$-orbits,
 it is a Poisson submanifold, with induced Poisson structure $\pi_{\Sigma}:=\Pi|_{\Sigma_{\varDelta}}$. 
 Because $\pi_{\Sigma}$ is $N_{\C}$-invariant, the quotient 
 map $p:\Sigma_{\varDelta} \to M_{\varDelta}$ pushes $\pi_{\Sigma}$ 
 to a Poisson structure $\pi_M$ on $M_{\varDelta}$ (cf. \cite{Ca}), whose fibres are coregular Poisson-Dirac submanifolds of $(\Sigma_{\varDelta},\pi_{\Sigma})$.
\end{proof}

When the positive bivector $\pi \in \wedge^2 \C^d$ is nondegenerate, the leaves of $(M_{\varDelta},\pi_M)$ are the orbits of the complex torus action, and so are finite in number. We refer to such manifolds as {\bf toric Poisson manifolds} (cf. \cite{Ca}). If complex conjugation is an anti-Poisson automorphism of $(\C^d,\pi)$ we say that $\pi$ is {\bf totally real}.

\subsection{Poisson submersions with Poisson fibres from varieties of full flags}

\vspace{0.5cm}

A closed subgroup $K$ of a Poisson-Lie group $(G,\pi_G)$  is a \emph{Poisson-Lie subgroup} if it is a Poisson submanifold of $(G,\pi_G)$; otherwise said, if $\pi_G$ is tangent to $K$, $\pi_G|_K=\pi_K \in \X^2(K)$, in which case $(K,\pi_K)$ becomes a Poisson-Lie group in its own right. 

In the following proposition, we look at different ways in which a closed subgroup interacts with the ambient Poisson-Lie group structure (cf. \cite[Section 4]{LW} and \cite[Proposition 2]{STS})

\begin{proposition}\label{pro : quotient by a subgroup}
 Let a Poisson-Lie group $(G,\pi_G)$ correspond to the $G$-invariant Manin triple 
 $(\mathfrak{d},\abk,\mathfrak{g},\mathfrak{h})$, and for a connected, closed subgroup $K\subset G$ with Lie algebra $\mathfrak{k} \subset \mathfrak{g}$, denote by $p:G \to M:=G/K$ the quotient map under the action
 \begin{align*}
  & K \acts G, & \alpha(k,g)=gk^{-1}.
 \end{align*}
 \begin{enumerate}[a)]
 \item $\pi_G$ is $K$-invariant $\Leftrightarrow$ $[\mathfrak{k},\mathfrak{h}]\subset \mathfrak{h}$, in which case $\pi_G$ vanishes along $K$.
 \item $K$ is a Poisson submanifold $\Leftrightarrow$ $\mathfrak{k}^{\circ} \subset \mathfrak{h}$ is an ideal $\Leftrightarrow$ $[\mathfrak{k},\mathfrak{k}^{\circ}]\subset \mathfrak{h}$.
  \item A Poisson structure $\pi_M$ on $M$ exists, for which the quotient map $p:(G,\pi_G) \to (M,\pi_M)$ is a 
Poisson submersion $\Leftrightarrow$ $\mathfrak{k}^{\circ} \subset \mathfrak{h}$ is a subalgebra. 
  \item If $K$ is a Poisson submanifold, then fibres of $p:G \to M$ are Poisson-Dirac submanifolds iff \[\mathrm{Ad}_G(\mathfrak{h}) \cap (\mathfrak{k} \oplus \mathfrak{k}^{\circ}) \subset \mathfrak{k}^{\circ}.\]
  In that case, the Poisson structures on fibres are all trivial exactly when \[\mathrm{Ad}_G(\mathfrak{h}) \cap (\mathfrak{k} \oplus \mathfrak{h}) \subset \mathfrak{h}.\]
 \end{enumerate}
\end{proposition}
\begin{proof}
 First note that the infinitesimal action $\mathfrak{k} \to \X(G)$ is given by the map $\varrho$ of \eqref{eq : dressing}, since
\begin{align*}
 \tfrac{d}{dt}g\exp(-tv)^{-1} = \tfrac{d}{dt}l_g(\exp(tv)) = (l_g)_*(v) = v^L_g = \varrho(v)_g.
\end{align*}If we let $V \subset TG$ stand for the vertical bundle of $p:G \to M$, then
\begin{align*}
& V = \{ \varrho(v)_g \ | \ (g,v) \in G \times \mathfrak{k}\}, & V^{\circ} = \{ \langle \theta^L_g,w\rangle \ | \ w \in \mathfrak{k}^{\circ} \}
\end{align*}
Indeed, $V$ is spanned by left-translates of $\mathfrak{k}$, and
\begin{align*}
 V^{\circ}_g & = \{ \langle \theta^L_g,w\rangle \ | \ \iota_{\varrho(v)_g}\langle \theta^L_g,w\rangle = \langle \iota_{\varrho(v)_g}\theta^L_g,w\rangle = \langle \mathrm{Ad}_{g^{-1}}\pr_{\mathfrak{g}}\mathrm{Ad}_{g}(v),w\rangle = 0, \ v \in \mathfrak{k} \} \\
 & = \{ \langle \theta^L_g,w\rangle \ | \ \langle v,w\rangle = 0, \ v \in \mathfrak{k} \}= \{ \langle \theta^L_g,w\rangle \ | \ w \in \mathfrak{k}^{\perp} \}\\
 & = \{ \langle \theta^L_g,w\rangle \ | \ w \in \mathfrak{k}^{\circ} \},
\end{align*}where we used the fact that $\mathfrak{k}^{\perp}=\mathfrak{g}\oplus \mathfrak{k}^{\circ}$, and that $\langle \theta^L,\mathfrak{g}\rangle=0$. Because the map $\epsilon : \mathfrak{d} \to \Gamma(\mathbb{T}G)$ of \eqref{eq : epsilon} is an algebra map, we have
\begin{align*}
 \epsilon[v,w] & = [\epsilon(v),\epsilon(w)] = [\varrho(v),\varrho(w)+\langle\theta^L,w\rangle] = [\varrho(v),\pi_G^{\sharp}\langle\theta^L,w\rangle+\langle\theta^L,w\rangle] 
\end{align*}for all $v \in \mathfrak{k}$ and $w \in \mathfrak{h}$, and therefore
\begin{align*}
 [\mathfrak{k},\mathfrak{h}] \subset \mathfrak{h} \qquad \Iff \qquad [\varrho(\mathfrak{k}),\Gamma(\mathrm{Gr}\ \pi_G)] \subset \Gamma(\mathrm{Gr}\ \pi_G),
\end{align*}which is the same as saying that the infinitesimal action $\varrho : \mathfrak{k} \to \X(G)$ is by Poisson automorphisms of $(G,\pi_G)$. Because $K$ is assumed to be connected, this is in turn equivalent to demanding that $\pi_G \in \X^2(G)$ be $K$-invariant. By multiplicativity of $\pi_G$, 
\begin{align*}
 & \pi_{G,gk^{-1}} = (l_g)_*(\pi_{G,k^{-1}})+(r_{k^{-1}})_*(\pi_{G,k}) = (r_{k^{-1}})_*(\pi_{G,k}), & (k,g) \in K \times G
\end{align*}this implies that $\pi_G$ vanishes along $K$, and hence that $(K,0) \acts (G,\pi_G)$ is a Poisson action, whose orbits are the fibres of $p:G \to M$. This proves a). Observe next that $K$ is a Poisson submanifold iff
\begin{align*}
 \pi_G^{\sharp}(N^*K) = \{ \varrho(w)_k \ | \ (k,w) \in K \times \mathfrak{k}^{\circ}\}
\end{align*}is trivial, which is tantamount to saying that $\mathrm{Ad}_K(\mathfrak{k}^{\circ}) \subset \mathfrak{h}$, or, equivalently, that $[\mathfrak{k},\mathfrak{k}^{\circ}] \subset \mathfrak{h}$ --- which by invariance of $\abk$ is yet equivalent to $[\mathfrak{k}^{\circ},\mathfrak{h}] \subset \mathfrak{k}^{\circ}$. This proves b). Moreover, because
\begin{align*}
 \epsilon(G \times \mathfrak{k}^{\circ}) = \{ \rho(w)_g+\la \theta^L_g,w\ra \ | \ (g,w) \in G \times \mathfrak{k}^{\circ}\} = \mathcal{R}_{\pi_G}(V^{\circ}),
\end{align*}it follows that
\begin{align*}
 \epsilon(G \times \mathrm{Gr}(\mathfrak{k})) = \mathcal{R}_{\pi_G}\mathrm{Gr}(V).
\end{align*}Again by invariance of $\abk$, we have that
\[
 \mathrm{Gr}(\mathfrak{k}) \ \ \text{is a subalgebra} \ \ \Leftrightarrow \ \ \mathfrak{k}^{\circ} \ \ \text{is a subalgebra} \ \ \Leftrightarrow \ \ [\mathfrak{k},\mathfrak{k}^{\circ}] \subset \mathrm{Gr}(\mathfrak{k}),
\]and therefore $\mathfrak{k}^{\circ}$ is a subalgebra exactly when $\mathcal{R}_{\pi_G}\mathrm{Gr}(V)$ is a Dirac structure on $G$. Because $\mathcal{R}_{\pi_G}\mathrm{Gr}(V)$ contains the vertical bundle, it is a basic Dirac structure according to \cite[Proposition 1]{PT1.75} --- that is, it is the pullback of a Dirac structure $L_M$ on $M$,
\begin{align*}
 \mathcal{R}_{\pi_G}\mathrm{Gr}(V) = p^!(L_M),
\end{align*}in which case
\begin{align*}
 L_M = p_!\mathcal{R}_{\pi_G}\mathrm{Gr}(V) = p_!\mathrm{Gr}(\pi_G).
\end{align*}Because $L_M$ arises from pushing forward a Dirac structure, it must itself be the graph of a Poisson structure $\pi_M$ on $M$, for which $p:(G,\pi_G) \rmap (M,\pi_M)$ is a Poisson map. This proves c).

Observe next that, if $(K,\pi_K)$ is a Poisson-Lie subgroup of $(G,\pi_G)$, then
\[
 \alpha : (K,-\pi_K) \times (G,\pi_G) \rmap (G,\pi_G)
\]is a Poisson action, whose orbits are the fibres of $p:G \to M$. Invoking Lemma \ref{lem:orbits}, we deduce that the fibres of $p$ have an induced Poisson structure exactly when the intersection
 \begin{align*}
  \pi_G^{\sharp}(V^{\circ}) \cap V & = \{ \varrho(v)_g \ | \ \exists \ w \in \mathfrak{k}^{\circ}, \ \varrho(v)_g=\varrho(w)_g \}\\
  & = \{ \varrho(v)_g \ | \ v \in \pr_{\mathfrak{g}}\ker\left({\varrho_g}|_{\mathfrak{k} \oplus \mathfrak{k}^{\circ}}\right)\}
 \end{align*}is trivial. This is equivalent to the condition that
 \begin{align*}
  (g, v,w) \in G \times \mathfrak{k} \times \mathfrak{k}^{\circ}, \qquad \mathrm{Ad}_g(v+w) \in \mathfrak{h} \quad \Longrightarrow \quad v=0,
 \end{align*}that is, that $\mathrm{Ad}_G(\mathfrak{h}) \cap (\mathfrak{k} \oplus \mathfrak{k}^{\circ}) \subset \mathfrak{k}^{\circ}$. Arguing in exactly the same manner we conclude that
 \begin{align*}
  \pi_G^{\sharp}(T^*G) \cap V & = \{ \varrho(v)_g \ | \ \exists \ w \in \mathfrak{h}, \ \varrho(v)_g=\varrho(w)_g \}
 \end{align*}is trivial --- that is, that the induced structures on fibres are trivial --- exactly when  $\mathrm{Ad}_G(\mathfrak{h}) \cap (\mathfrak{k} \oplus \mathfrak{h}) \subset \mathfrak{h}$.
\end{proof}

On a compact connected semisimple Lie group a choice of maximal torus and root order gives rise to a ``standard'' Poisson-Lie group structure which descends to the {\bf manifolds of full flags} \cite{LW}:

\begin{proposition}\label{pro : Lu-Weinstein}
Let a connected compact Lie group $G$ be equipped with its ``standard'' Poisson structure $\pi_G$ and let $(M,\pi_M)$ be the manifold of full flags. Then $p:(G,\pi_G) \to (M,\pi_M)$
 is a Poisson submersion with Poisson fibres, and the Poisson structures on fibres are all trivial.
\end{proposition}
\begin{proof}
Let us fix a maximal torus $T \subset G$, with Lie algebra $\mathfrak{t}$. Then $\mathfrak{t}_{\C}$ is a Cartan subalgebra for the complex semisimple Lie algebra $\mathfrak{d}:=\mathfrak{g}_{\C}$; hence there is a splitting of vector spaces
\begin{align*}
 \mathfrak{d} = \mathfrak{t}_{\C} \oplus \bigoplus_{\alpha \in \Delta}\mathfrak{d}_{\alpha},
\end{align*}where $\Delta$ is the set of all roots $\alpha \in \mathfrak{t}_{\C}^*\diagdown \{0\}$.

The Killing form $\mathrm{B}_{\mathfrak{d}} \in S^2\mathfrak{d}^*$ of $\mathfrak{d}$ is an invariant, nondegenerate, symmetric bilinear pairing. It is nondegenerate on $\mathfrak{t}_{\C} \times \mathfrak{t}_{\C}$ and on $\mathfrak{d}_{\alpha} \times \mathfrak{d}_{-\alpha}$  --- on $\mathfrak{d}_{\alpha} \times \mathfrak{d}_{\beta}$ it vanishes if $\alpha + \beta \neq 0$. Moreover, there is a splitting of real vector spaces
 \begin{align*}
 & \mathfrak{t}_{\C} = \mathfrak{t} \oplus \mathfrak{a}, & \mathfrak{a}:=\{ v  \in \mathfrak{t}_{\C} \ | \ \alpha \in \Delta \ \Rightarrow \ \alpha(v) \in \R\},
 \end{align*}and $\mathrm{B}_{\mathfrak{d}}|_{\mathfrak{t} \times \mathfrak{t}}$ is negative-definite and $\mathrm{B}_{\mathfrak{d}}|_{\mathfrak{a} \times \mathfrak{a}}$ is positive-definite. Let us now fix a set of positive roots $\Delta^+ \subset \Delta$, and write $\Delta = \Delta^+ \coprod \Delta^-$, where $\Delta^-:=-\Delta^+$. Define
 \begin{align*}
  & \mathfrak{n}_{\pm}:=\bigoplus_{\alpha \in \Delta^{\pm}}\mathfrak{d}_{\alpha}, & \mathfrak{b}_{\pm}:= \mathfrak{t}_{\C}\oplus \mathfrak{b}_{\pm}.
 \end{align*}Then $\mathfrak{n}_{\pm}$ are nilpotent Lie algebras, and $\mathfrak{b}_{\pm}$ are solvable Lie algebras; in fact, $[\mathfrak{b}_{\pm},\mathfrak{b}_{\pm}]=\mathfrak{n}_{\pm}$. Moreover, the splittings $\mathfrak{b}_{\pm}=\mathfrak{t}\oplus \mathfrak{a} \oplus \mathfrak{n}_{\pm}$ can be described in terms of the spectrum of the adjoint map $\mathrm{ad}(v):\mathfrak{d} \to \mathfrak{d}$, in the sense that, for $v \in \mathfrak{b}_{\pm}$:
 \begin{enumerate}[S1)]
  \item $v \in \mathfrak{t}\oplus \mathfrak{n}_{\pm} \ \Leftrightarrow \ \mathrm{Spec}(\mathrm{ad}(v)) \subset i\R$;
  \item $v \in \mathfrak{a}\oplus \mathfrak{n}_{\pm} \ \Leftrightarrow \ \mathrm{Spec}(\mathrm{ad}(v)) \subset \R$;
  \item $v \in \mathfrak{n}_{\pm} \ \Leftrightarrow \ \mathrm{Spec}(\mathrm{ad}(v)) = \{0\}$.
 \end{enumerate}
Define further 
 \begin{align*}
  & \mathfrak{h}:= \mathfrak{a} \oplus \mathfrak{n}_{+}, & \langle v,w\rangle:=\Im \mathrm{B}_{\mathfrak{d}}(v,w).
 \end{align*}Then $(\mathfrak{d},\abk,\mathfrak{g},\mathfrak{h})$ is a Manin triple, and because $G$ is connected, there are no choices in the representation $\mathrm{Ad}:G \acts \mathfrak{d}$, so we may regard $(\mathfrak{d},\abk,\mathfrak{g},\mathfrak{h})$ as $G$-invariant Manin triple. The Poisson structure $\pi_G \in \X^2(G)$ that corresponds to it is the ``standard'' or \emph{Lu-Weinstein} Poisson structure on $G$. By item a) in Proposition \ref{pro : quotient by a subgroup}, the fact that
 \begin{align*}
  [\mathfrak{t},\mathfrak{h}] = \mathfrak{n}_+ \subset \mathfrak{h}
 \end{align*}implies that $p:G \to G/T$ pushes $\pi_G$ to a Poisson structure $\pi_M$ on $M=G/T$. On the other hand, suppose $(g,x,y,z,w) \in G \times \mathfrak{a} \times \mathfrak{n}_{\pm} \times \mathfrak{t} \times \mathfrak{n}_{\pm}$ is such that
 \begin{align*}
  \mathrm{Ad}_g(x+y) = z+w.
 \end{align*}By \textrm{S1)}, $\mathrm{Spec}(\mathrm{ad}(z+w)) \subset i\R$ and by \textrm{S2)},
 \begin{align*}
 \mathrm{Spec}(\mathrm{ad}\mathrm{Ad}_g(x+y)) = \mathrm{Spec}(\mathrm{ad}(x+y)) \subset \R.
 \end{align*}Thus \textrm{S3)} implies that $x=0$ and $z=0$. That is,
 \begin{align*}
  \mathrm{Ad}_G(\mathfrak{h}) \cap (\mathfrak{t} \oplus \mathfrak{n}_+) \subset \mathfrak{n}_+,
 \end{align*}which implies by item d) in Proposition \ref{pro : quotient by a subgroup} that 
 the fibres of $p:G \to M$ are Poisson-Dirac submanifolds. In fact, by the same argument one deduces that
 \begin{align*}
  \mathrm{Ad}_G(\mathfrak{h}) \cap (\mathfrak{t} \oplus \mathfrak{h}) \subset \mathfrak{h},
 \end{align*}which is to say that all Poisson structures on fibres are trivial.
\end{proof}

\subsection{Poisson structures on associated bundles}

Recall that if $p:P \to M$ is a right principal $G$-bundle, and $G$ acts on the left on a manifold $X$, we refer to the quotient of the free left action
\begin{align*}
 & G \curvearrowright P \times X, & g \cdot (y,x):=(yg^{-1},gx)
\end{align*}as the {\bf associated bundle}
\begin{align*}
 & p:\Sigma:=P\times_GX \to M, & p[y,x]=p(y).
\end{align*}

The first observation is that, when $P$ is equipped with a $G$-invariant Poisson structure, the associated bundle $\Sigma$ has an induced Poisson structure --- and certain good properties of the principal bundle are inherited by its associated bundle:

\begin{lemma}\label{lem : associated inherits pencils}
 Let a right principal $G$-bundle $p:P\to M$ be endowed with a $G$-invariant Poisson structure $\pi_P$, 
 and let $G$ act on the left of a Poisson manifold $(X,\pi_X)$ by Poisson diffeomorphisms. Then there exist Poisson structures $\pi_M$ on $M$ and $\pi_{\Sigma}$ on $\Sigma:=P \times_GX$, for which
 \[
  \xymatrix{
  (P,\pi_P) \times (X,\pi_X) \ar[r]^{\phantom{12345}q} \ar[d]_{\pr_1} & (\Sigma,\pi_{\Sigma}) \ar[d]^p\\
  (P,\pi_P) \ar[r]_p & (M,\pi_M)
  }
 \]are all Poisson submersions. If $p:(P,\pi_P) \to (M,\pi_M)$ is a Poisson submersion with Poisson fibres, or an orthogonal pencil, then so is $p:(\Sigma,\pi_{\Sigma}) \to (M,\pi_M)$.
\end{lemma}
\begin{proof}
  Because $\pi_P\in \X^2(P)$ is $G$-invariant there is 
  a unique Poisson structure  $\pi_M \in \X^2(M)$,
 such that $p:(P,\pi_P) \to (M,\pi_M)$ is a Poisson submersion.  Likewise, the Poisson structure $(\pi_P,\pi_X)$ on $P \times X$ is invariant under the diagonal action
 \begin{align*}
  & G \acts P \times X, & g(y,x)=(yg^{-1},gx),
 \end{align*}and a unique Poisson structure $\pi_{\Sigma}$ exists on $\Sigma$, for which $q:(P,\pi_P) \times (X,\pi_X) \to (\Sigma,\pi_{\Sigma})$ is a Poisson submersion. In fact, $(\pi_P,0)$ and $(0,\pi_X)$ are both $G$-invariant, and if we denote by $\pi_H,\pi_V \in \X^2(\Sigma)$ the respective Poisson bivectors which correspond to them, we have a splitting into commuting Poisson structures
 \begin{align*}
  \pi_{\Sigma}=\pi_H+\pi_V,
 \end{align*}in which $\pi_V$ is tangent to the fibres of $p:\Sigma \to M$. Suppose on the one hand that $p:(P,\pi_P) \to (M,\pi_M)$ is a Poisson submersion with Poisson fibres. Then
 \begin{align*}
  \pi_{\Sigma}^{\sharp}(V_{\Sigma}^{\circ}) & = q_*(\pi_P^{\sharp},\pi_X^{\sharp})q^*(V_{\Sigma}^{\circ})  = q_*(\pi_P^{\sharp},\pi_X^{\sharp})(V_P^{\circ} \times X) = q_*(\pi_P^{\sharp}(V_P^{\circ}) \times X)
 \end{align*}is a vector bundle, which meets $V_{\Sigma}$ trivially because $q_*(\pi_P^{\sharp}(\xi),0) \in V_{\Sigma}$ implies that $\pi_P^{\sharp}(\xi) \in V_P$. Hence $p:(\Sigma,\pi_{\Sigma}) \to (M,\pi_M)$ would be a Poisson submersion with Poisson fibres as well. Suppose on the other hand that $p:(P,\pi_P) \to (M,\pi_M)$ is an orthogonal pencil. By Theorem \ref{thm : orthogonal splittings}, this is equivalent to the vertical Poisson structures on fibres assembling into a smooth, vertical Poisson structure $\widetilde{\pi}_V \in \Gamma(\wedge^2V_{\Sigma})$ --- in which case $\pi_V:=q_*\widetilde{\pi}_V+\pi_X \in \Gamma(\wedge^2V_{\Sigma})$ is a smooth, vertical Poisson structure on $\Sigma$, which induces on fibres the same Poisson structure as $\pi_{\Sigma}$. Hence $p:(\Sigma,\pi_{\Sigma}) \to (M,\pi_M)$ is  an orthogonal pencil. 
\end{proof}

Even when the principal bundle itself fails to be sufficiently well-behaved, one may sometimes impose conditions on the Poisson structure $\pi_P$ and the action $G \curvearrowright (X,\pi_X)$ to require better behavior of its associated bundle; e.g.,

\begin{lemma}\label{lem : self-made pencil}
 Let a right principal $G$-bundle $p:P\to M$ be endowed with a $G$-invariant Poisson structure $\pi_P$, 
 and let $G$ act on the left of a Poisson manifold $(X,\pi_X)$ by Poisson diffeomorphisms. Let $W_G \subset V_P \times TX$ denote the vertical bundle to $q:P \times X \to \Sigma$, and $V_P$ the vertical bundle to $p:P \to M$. If $\pi_P^{\sharp}:W_G^{\circ} \to TP \times X$ meets $V_P \times X$ in a vector bundle, then $p:(\Sigma,\pi_{\Sigma}) \to (M,\pi_M)$ is an orthogonal pencil.
\end{lemma}
\begin{proof}
  Observe that by the characterization of Theorem \ref{thm : orthogonal splittings}, $p:(\Sigma,\pi_H) \to (M,\pi_M)$ is an orthogonal pencil iff $\mathrm{Gr}(\pi_H) \cap (V_{\Sigma}\oplus T^*\Sigma) + V_{\Sigma}^{\circ}$ is a vector bundle.  Because $q$ is a Poisson submersion, this is the case exactly when
 \begin{align*}
  q^!\mathrm{Gr}(\pi_H) \cap \left((V_P \times TX) \oplus (T^*P \times T^*X)\right) + V_P^{\circ} \times X 
 \end{align*}is a vector bundle --- and since $q^!\mathrm{Gr}(\pi_H)=\mathcal{R}_{\pi_H}\mathrm{Gr}(W_Q)$, that is implied by the hypothesis that
 \[
  (\pi_P^{\sharp},0)^{-1}(V_P \times X) \subset W_Q^{\circ}
 \]is a vector bundle. 
\end{proof}

Let us (tentatively) say that a Poisson submersion with Poisson fibres $p:(\Sigma,\pi_{\Sigma}) \to (M,\pi_M)$ is {\bf strongly locally trivial} if every point $x \in M$ lies in an open set $U \subset M$, over which a diffeomorphism $\phi : U \times X \diffto p^{-1}(U)$ exists, such that
\[
 \mathrm{Gr}(\pi_M)|_U \times \mathrm{Gr}(\pi_X) = \phi^!\mathrm{Gr}(\pi_{\Sigma}).\]Any two fibres of a strongly locally trivial Poisson submersion are Poisson diffeomorphic. It is also worthwhile to consider the case of Poisson submersions in which nearby fibres are \emph{gauge-equivalent}: we call a Poisson submersion with Poisson fibres {\bf locally trivial} if
\begin{align}\label{eq : gauge condition on fibres}
 \mathcal{R}_{\omega}\left(\mathrm{Gr}(\pi_M)|_U \times \mathrm{Gr}(\pi_X)\right) = \phi^!\mathrm{Gr}(\pi_{\Sigma}) 
\end{align}
for a closed two-form $\omega \in \Omega^2(U \times X)$ for which
 \[
 \id + (0,\pi_X)^{\sharp}\omega : U \times TX \rmap U \times TX
 \]is an isomorphism. Note that the latter condition is tantamount to requiring that
 $p:(\Sigma,\pi_{\Sigma}) \to (M,\pi_M)$ has Poisson fibres.

\begin{remark}\normalfont
    Strongly locally trivial Poisson submersions are locally trivial. In favorable circumstances, one may employ the Moser trick \cite[Lemma 4]{PT1} to attempt to promote a local trivialization into a strong one. This strategy (or any strategy, for that matter) may fail, however: below we present an example which is locally trivial, but not strongly locally trivial for completeness reasons:  
\end{remark}

\begin{example}\normalfont
    \noindent \emph{An example of a locally trivial Poisson submersion with Poisson fibres which is not strongly locally trivial.} Let $M=\mathbb{R}_+$ have coordinate $t$, and let
    \begin{align*}
        & X \subset \mathbb{R}^2, & X = \{(x_1,x_2) \ | \ x_1^2+x_2^2<1\}.
    \end{align*}Consider on $M$ the trivial Poisson structure $\pi_M=0$, and on $X$ the standard Poisson structure $\pi_X = \tfrac{\partial}{\partial x_1} \wedge \tfrac{\partial}{\partial x_2}$. Consider the closed two-form
    \begin{align*}
        & \omega:=-\dd(t\alpha), & \alpha:=\tfrac{1}{2}(x_1\dd x_2-x_2\dd x_1),
    \end{align*}on $\Sigma:=M \times X$, equipped with the Poisson structure
    \begin{align*}
        \mathrm{Gr}(\pi_{\Sigma}) & :=\mathcal{R}_{\omega}(\mathrm{Gr}(\pi_M) \times \mathrm{Gr}(\pi_X)) \\
        & = \mathrm{Gr}(\tfrac{1}{1+t}\pi_X).
    \end{align*}
    Then the canonical projection
    \begin{align*}
        p : (\Sigma,\pi_{\Sigma}) \to (M,0)
    \end{align*}
    is a Poisson submersion with Poisson fibres, which is locally trivial by construction. 
    
    However, no strong local trivialization can exist. For example, suppose a strong local trivialization would exist around, say $t=1$. Then for some open set $1 \in U \subset M$, there would exist a smooth embedding $\varphi$,
    \begin{align*}
        \varphi : U \times X \diffto \Sigma,
    \end{align*}
     such that $p\varphi=p$, and
    \begin{align*}
        \varphi : (U,\pi_U) \times (X,\pi_X) \to (\Sigma, \tfrac{1}{1+t}\pi_X)
    \end{align*}is a Poisson map. A such map would induce, for all $t \in U$, Poisson diffeomorphisms
    \begin{align*}
        \varphi_t : (X,\pi_X) \diffto (X,\tfrac{1}{1+t}\pi_X),
    \end{align*}
    which is impossible since the symplectic areas change.
\end{example}

\begin{example}\label{ex: complete}\normalfont
 A Poisson submersion $p:(\Sigma,\pi_{\Sigma}) \to (M,\pi_M)$ is {\bf complete} 
 if the Hamiltonian vector field of a function $f \circ p \in C^{\infty}(\Sigma)$ is 
 complete if that of $f \in C^{\infty}(M)$ is complete. It follows
 from Theorem \ref{thm : restriction of PD submersion over leaf} that a complete Poisson
 submersion with Poisson fibres $p:(\Sigma,\pi_{\Sigma}) \to (M,\pi_M)$ restricts over each leaf $\mathrm{S}_M(x)$ of $(M,\pi_M)$ to
 a strongly locally trivial
 Poisson submersion\[p:(p^{-1}\mathrm{S}_M(x),\pi_{p^{-1}\mathrm{S}_M(x)}) \to (\mathrm{S}_M(x),\pi_{\mathrm{S}_M(x)}).\]
 Indeed, recall that a Poisson submersion is coupling exactly when its horizontal foliation is an Ehreshmann connection. In such a situation  there is a mandatory strategy to attempt to produce a strong local trivialization around any given fibre: To integrate the flows of Hamiltonian vector fields of (appropriate) functions pulled back from the base. The success of the construction for every fibre is equivalent to the completeness of the (coupling) Poisson submersion.
\end{example}

\begin{lemma}
 A locally trivial Poisson submersion is an orthogonal pencil. A strongly locally trivial Poisson submersion is almost-coupling. 
\end{lemma}
\begin{proof}
 Let $p \colon (\Sigma,\pi_{\Sigma}) \to (M,\pi_M)$ be a locally trivial Poisson submersion. As we already observed, $p$ has Poisson fibres, so in order to show that it is an orthogonal pencil, it suffices by Theorem \ref{thm : orthogonal splittings} to check that the Poisson-Dirac structures on fibres vary smoothly. This is a local matter in $M$, which we need only check over a trivialization $\Phi := \phi_*\mathcal{R}_{\omega} : (U,\pi_M) \times (X,\pi_X) \diffto (p^{-1}(U),\pi_{\Sigma})$. But
 \begin{align*}
     & \mathcal{R}_{\omega}\left(\mathrm{Gr}(\pi_M)|_U \times \mathrm{Gr}(\pi_X)\right) = \phi^!\mathrm{Gr}(\pi_{\Sigma}), & \id + (0,\pi_X)^{\sharp}\omega : U \times TX \diffto U \times TX,
 \end{align*}exhibits the fibres of $p$ as gauge-transformations of $\pi_X$ by the restriction of the smooth form $\omega$ to fibres, which is clearly smooth. Thus a locally trivial Poisson submersion is an orthogonal pencil. 
 
 When the submersion is strongly locally trivial (so that $\omega$ may be chosen to vanish identically), then Ehresmann connections on $\mathrm{pr}:U \times X \to U$ which make it into an almost-coupling submersion form a non-empty convex set, and so a global Ehresmann connection can be built out of a partition of unity subordinated to a trivializing cover, which turns $p:(\Sigma,\pi_{\Sigma}) \to (M,\pi_M)$ into an almost-coupling submersion.
\end{proof}

Let $\pi_P$ be an invariant Poisson structure on the total space of the right principal $G$-bundle $p:P \to M$. Then $p$ pushes $\pi_P$ to a Poisson structure $\pi_M$ on $M$, i.e.  $p \colon (P,\pi_{P}) \to (M,\pi_M)$ is a Poisson submersion. Such a ``Poisson principal bundle" $p \colon (P,\pi_{P}) \to (M,\pi_M)$ is {\bf equivariantly} locally trivial if around each $x \in M$ there is an open set $U \subset M$, over which there exists a principal bundle trivialization $\phi : U \times X \diffto p^{-1}(U)$ and a $G$-invariant closed two-form $\omega \in \Omega^2(U \times G)^G$ for which
 \begin{align*}
     \mathcal{R}_{\omega}\left(\mathrm{Gr}(\pi_M)|_U \times \mathrm{Gr}(\pi_G)\right) = \phi^!\mathrm{Gr}(\pi_{P})
 \end{align*}and
 \begin{align*}
     \id + (0,\pi_G)^{\sharp}\omega : U \times TG \diffto U \times TG.
 \end{align*}is an isomorphism. 

\begin{remark}\label{rem : principal trivializations come for free in the symplectic case}\normalfont
   A symplectic principal bundle with symplectic fibres is always equivariantly locally trivial: the symplectic orthogonal to fibres defines a \ul{complete} principal (symplectic) connection\cite[Proposition 6.6]{CW}.
 \end{remark}

There are Poisson principal bundles which are locally trivial but not equivariantly so, as the example below illustrates.

 \begin{example}\normalfont
 Consider the trivial principal $G$-bundle  $p:(G\times \R,\pi_{P}) \to \R$ endowed with a $G$-invariant vertical Poisson structure. That $\pi_{P}$ be strongly locally trivial is equivalent to $\pi_{p^{-1}(x)}$ having nearby Poisson diffeomorphic fibers (by means of a smooth 1-parameter family of Poisson diffeomorphism). Equivariant strong local triviality amounts to realizing such diffeomorphisms by Lie group automorphisms.

     Let $G:=\mathrm{Aff}(\R)\times \mathrm{Aff}(\R)$, where
     $\mathrm{Aff}(\R)$ is the group of affine motions of the line. Consider the
     the 1-parameter family of linear 2-forms on the Lie algebra of $G$
     \begin{align*}
         & e_1^*\wedge e_2^*+x e_1^*\wedge e_3^*+e_3^*\wedge e_4^* \in \wedge^2 \mathfrak{g}^*, & x\in \R,
     \end{align*}where
     \begin{align*}
         & e_1=\begin{pmatrix} 0 & 0\\ 0 & 1\end{pmatrix}, & e_2=\begin{pmatrix} 0 & 0\\ 1 & 0\end{pmatrix}\in \mathrm{aff}(\R)
     \end{align*}
     and $e_3=e_1$, $e_4=e_2$ belong to a second copy of $\mathrm{aff}(\R)$. Let $\pi_P$ be the Poisson structure corresponding to the 1-parameter family of right-invariant symplectic forms which integrates the previous family of linear 2-forms.
     For different positive values of $x$, the  right-invariant symplectic forms induced on $G$ cannot be related by a Lie group automorphism \cite[Proposition 2.4]{Ov}. However, the Moser trick in logarithmic coordinates produces a symplectic trivialization.
 \end{example}

\begin{lemma}\label{lem : SLT principal has SLT associated}
 Bundles associated to strongly, equivariantly locally trivial principal bundles are strongly locally trivial.
\end{lemma}

\begin{proof}
 Let $\pi_P$ be an invariant Poisson structure on the total space of the right principal $G$-bundle $p:P \to M$ which makes it into a Poisson submersion with Poisson fibres. If $\pi_P$ is equivariantly strongly locally trivial, around each point $x \in M$ there is a local bundle trivialization $\phi : U \times G \to P$ and a $G$-invariant Poisson structure $\pi_G$ on $G$, such that $(\pi_M|_U,\pi_G)=\phi^*\pi_P$. Then $(\phi,\mathrm{id}_X) : U \times G \times X \to P \times X$ induces a local trivialization $\overline{\phi}: U \times X \to \Sigma$ of the associated bundle $\Sigma=P \times_GX$, making
  \begin{equation}\label{eq:associated trivialization}
   \xymatrix{
   U \times G \times X \ar[d]_{(\mathrm{id}_M,\alpha)}\ar[rr]^{(\phi,\mathrm{id}_X)} && P \times X \ar[d]^q\\
   U \times X \ar[rr]_{\overline{\phi}} && \Sigma
   }
  \end{equation} commute, where $\alpha : G \times X \to X$ is the action map. Now, $\pi_G$ being invariant under the right action of $G$ on itself means that $\pi_{G,g} = (r_g)_*(\pi_{G,e})$ for all $g \in G$. Observe that
\begin{align*}
 & \alpha_*(v^L_g,u) = -\rho(v) + (\alpha_g)_*(u), & v \in \mathfrak{g}, \quad u \in TX,
\end{align*}where $v^L \in \X(G)$ denotes the left-invariant vector field corresponding to $v$, and $\rho(v) \in \X(X)$ denotes the infinitesimal action of $v$. This implies that
\begin{align}\label{eq: new fiber factor}
 & \alpha : (G,\pi_{G}) \times (X,\pi_X) \rmap (X,\widetilde{\pi}_X), & \widetilde{\pi}_X:=\wedge^2\rho(\pi_{G,e}) + \pi_X
\end{align}is a Poisson map, and therefore
\begin{align*}
 & \overline{\phi} : (U,\pi_M) \times (X,\widetilde{\pi}_X) \diffto (\Sigma,\pi_{\Sigma})
\end{align*}is a Poisson diffeomorphism, whence $\pi_{\Sigma}$ is strongly locally trivial.
\end{proof}

In yet another tentative flavor of local triviality of a Poisson submersion with Poisson fibres $p:(\Sigma,\pi_{\Sigma}) \to (M,\pi_M)$, one could require that it have a {\bf locally trivial foliation} --- namely, that around each point in $M$ there exist a trivialization $\phi : U \times X \diffto p^{-1}(U)$ with the property that
\[
 \mathrm{Gr}(\pi_M)|_U \times \mathrm{Gr}(\pi_X) \quad \text{and} \quad \phi^!\mathrm{Gr}(\pi_{\Sigma}) \quad \text{induce the same singular foliation on} \ U \times X.
 \]Clearly, locally trivial Poisson submersions have locally trivial foliation.

\begin{example}\label{ex : toy example of CP1}\normalfont 
\emph{An example of a Poisson submersion with Poisson fibres which has locally trivial foliation, and yet is not locally trivial.} Consider the submersion
\begin{align*}
 p:\Sigma=\mathbb{C}^2 \diagdown \{0\} \rmap \mathbb{C}P^1=M
\end{align*}of Example \ref{ex : coregular submersion with discontinuous Poisson structure on fibres}, but now equipped with the Poisson structure
\begin{align*}
 \pi_{\Sigma} = \tfrac{1}{2}(\mathscr{E}_0 \wedge \mathscr{V}_0+\mathscr{E}_1 \wedge \mathscr{V}_1).
\end{align*}This bivector is invariant under the $\mathbb{C}^{\times}$-action, and hence $p$ pushes $\pi_{\Sigma}$ to a Poisson structure $\pi_M$ on $M$. Consider the diagram of open embeddings
\begin{align*}
 & \xymatrix{
 \mathbb{C} \times \mathbb{C}^{\times} \ar[r]^{\phi_0} \ar[d]_{\pr_1} & \Sigma \ar[d]^p & \mathbb{C} \times \mathbb{C}^{\times} \ar[l]_{\phi_1} \ar[d]^{\pr_1}\\
 \mathbb{C} \ar[r] & M & \mathbb{C} \ar[l]
 }
  & \xymatrix{
 (z_0,z_1) \ar[r] \ar[d] & (z_1,z_0z_1) \ar[d] & (z_0z_1,z_1) \ar[d] & (z_0,z_1) \ar[l] \ar[d]\\
 z_0 \ar[r] & [1:z_0] & [z_0:1] & z_0 \ar[l]
 }
\end{align*}Then both $\phi_i$ are Poisson embeddings for the Poisson structure
\begin{align*}
 & \phi_i : (\mathbb{C} \times \mathbb{C}^{\times},\Pi) \rmap (\Sigma,\pi_{\Sigma}), & \Pi = \left(\mathscr{E}_0 -\tfrac{1}{2}\mathscr{E}_1\right)\wedge \left(\mathscr{V}_0 -\tfrac{1}{2}\mathscr{V}_1\right) +\tfrac{1}{4}\left(\mathscr{E}_1 \wedge \mathscr{V}_1\right)
\end{align*} From this, we read off that $\pr_1:(\mathbb{C} \times \mathbb{C}^{\times},\Pi) \to (\mathbb{C},\mathscr{E}_0 \wedge \mathscr{V}_0)$ is a Poisson submersion with Poisson fibres $(\mathbb{C}^{\times},\tfrac{1}{4}\mathscr{E}_1 \wedge \mathscr{V}_1)$, and that 
\begin{align}\label{eq : not locally trivial but locally trivial foliation}
 \mathrm{Gr}(\Pi) \qquad \text{and} \qquad \mathrm{Gr}(\mathscr{E}_0 \wedge \mathscr{V}_0) \times \mathrm{Gr}(\tfrac{1}{4}\mathscr{E}_1 \wedge \mathscr{V}_1)   
\end{align}
have the same singular foliation. From this it follows that $p:(\Sigma,\pi_{\Sigma}) \to (M,\pi_M)$ is a Poisson submersion with Poisson fibres with locally trivial foliation. Note that both $(M,\pi_M)$ and $(\Sigma,\pi_{\Sigma})$ have exactly three symplectic leaves. We claim that this Poisson submersion with Poisson fibres is not locally trivial around the fibres of $p$ over the singular points of $(M,\pi_M)$ --- that is, around the two singular leaves of $(\Sigma,\pi_{\Sigma})$. Indeed, suppose by contradiction that there exist:
\begin{itemize}
    \item an open neighborhood $U \subset \C$ of $0$;
    \item a closed two-form $\omega \in \Omega^2(U \times \C^{\times})$;
    \item a diffeomorphism $\varphi : U \times \C^{\times} \diffto U \times \C^{\times}$ of the form
    \begin{align*}
       &  \varphi(x_0,y_0,x_1,y_1)=(x_0,y_0,u_1,v_1), & (x_0,y_0) \in U, \ \ (x_1,y_1) \in \C^{\times},
    \end{align*}
\end{itemize}
with the property that
\begin{align}\label{eq : not locally trivial}
    \mathcal{R}_{\omega}\mathrm{Gr}(\mathscr{E}_0 \wedge \mathscr{V}_0+\tfrac{1}{4}\mathscr{E}_1 \wedge \mathscr{V}_1) = \varphi^!\mathrm{Gr}(\Pi).
\end{align}Because
\begin{align*}
\mathrm{Gr}(\Pi)|_{U \diagdown \{0\} \times \C^{\times}} & = \mathrm{Gr}(\sigma_1)\\
    \mathrm{Gr}(\mathscr{E}_0 \wedge \mathscr{V}_0+\tfrac{1}{4}\mathscr{E}_1 \wedge \mathscr{V}_1)|_{U \diagdown \{0\} \times \C^{\times}} & = \mathrm{Gr}(\sigma_2)
\end{align*}for symplectic forms $\sigma_1,\sigma_2 \in \Omega^2(U \diagdown \{0\})$. Condition \eqref{eq : not locally trivial} then reads:
\begin{align*}
    \omega = \varphi^*\sigma_1 - \sigma_2.
\end{align*}Note that
\begin{align}\label{eq : unbounded contraction}
\la \omega,\tfrac{\partial}{\partial x_0}\wedge \tfrac{\partial}{\partial y_0}\ra =  \tfrac{2}{(x_0^2+y_0^2)(u_1^2+v_1^2)}(f),
\end{align}
where
\begin{align*}
f  = & -x_0\left(v_1(\tfrac{\partial v_1}{\partial x_0}+\tfrac{\partial u_1}{\partial y_0})+u_1(\tfrac{\partial v_1}{\partial y_0}+\tfrac{\partial u_1}{\partial x_0})\right)+
y_0\left(u_1(\tfrac{\partial v_1}{\partial x_0}+\tfrac{\partial u_1}{\partial y_0})-v_1(\tfrac{\partial v_1}{\partial y_0}+\tfrac{\partial u_1}{\partial x_0})\right)\\
 & -2(x_0^2+y_0^2)\left(\tfrac{\partial u_1}{\partial x_0}\tfrac{\partial v_1}{\partial y_0}-\tfrac{\partial u_1}{\partial y_0}\tfrac{\partial v_1}{\partial x_0}\right)-\tfrac{1}{2}(u_1^2+v_1^2).
\end{align*}
We thus see that
\begin{align*}
    \lim_{(x_0,y_0,x_1,y_1) \to (0,0,1,0)}f = -\tfrac{1}{2}
\end{align*}implies that
\begin{align*}
    \lim_{(x_0,y_0,x_1,y_1) \to (0,0,1,0)}\la \omega,\tfrac{\partial}{\partial x_0}\wedge \tfrac{\partial}{\partial y_0}\ra = \lim_{(x_0,y_0,x_1,y_1) \to (0,0,1,0)}\tfrac{2}{(x_0^2+y_0^2)(u_1^2+v_1^2)}f = -\infty.
\end{align*}This shows that \eqref{eq : unbounded contraction} cannot extend to a smooth function on the whole $U \times \C^{\times}$, and therefore no such $\omega$ can exist (on the whole $U \times \C^{\times}$). Therefore, the Poisson submersion with Poisson fibres $p:(\Sigma,\pi_{\Sigma}) \to (M,\pi_M)$ is not locally trivial, in spite of having a locally trivial foliation.
\end{example}

 \begin{example}\normalfont \emph{An example of a bundle associated to locally trivial principal bundle whose foliation is not locally trivial.} The translation action
\begin{align*}
 & G=\mathbb{R}^2 \curvearrowright P=\{(y_1,y_2,y_3) \in \mathbb{R}^3 \ | \ y_3>0\}, & (a,b)\cdot (y_1,y_2,y_3) = (a+y_1,b+y_2,y_3),
\end{align*}turns 
\begin{align*}
 & p:P \to M=\mathbb{R}_+, & p(y_1,y_2,y_3)=y_3
\end{align*}into a principal $G$-bundle. The Poisson structure
\begin{align*}
 \pi_P=-y_3\tfrac{\partial}{\partial y_1}\wedge \tfrac{\partial}{\partial y_2}
\end{align*}on $P$ is $G$-invariant, and $p:(P,\pi_P) \to (M,0)$ is a locally trivial Poisson submersion with Poisson fibres. Let $G$ act on the Poisson manifold
\begin{align*}
 & X = \mathbb{R}^2, & \pi_X=\tfrac{\partial}{\partial x_1}\wedge \tfrac{\partial}{\partial x_2}
\end{align*}by $(a,b) \cdot (x_1,x_2) = (a+x_1,b+x_2)$. The submanifold
\begin{align*}
    & \Sigma \subset P \times X, & \Sigma = \{(0,0)\} \times \mathbb{R}_{+} \times \mathbb{R}^2
\end{align*} is a full slice to the action of $G$, and is therefore identified with the quotient $\Sigma = P \times_G X$. The associated Poisson structure is
\begin{align*}
 \pi_{\Sigma}=(1-y_3)\tfrac{\partial}{\partial x_1}\wedge \tfrac{\partial}{\partial x_2},
\end{align*}
which does not have locally trivial foliation around $y_3=1$.
\end{example}
Next, we observe that 

 \begin{proposition}\label{pro : zero or isotropic fiber associated}
  Let a right principal $G$-bundle $p:P\to M$ be endowed with a $G$-invariant Poisson structure $\pi_P$, 
  which makes it into a Poisson submersion with Poisson fibres. If $G$ acts by Poisson diffeomorphisms on a Poisson manifold $(X,\pi_X)$, with orbits contained in the symplectic leaves of $(X,\pi_X)$, and either:
  \begin{enumerate}[a)]
   \item The Poisson structures on the fibres of $p:P \to M$ are all trivial, or
   \item The orbits of $G\curvearrowright X$ are isotropic submanifolds of the symplectic leaves of $(X,\pi_X)$,
  \end{enumerate}then the induced Poisson structure $\pi_{\Sigma}$ on the associated bundle $\Sigma=P \times_GX \to M$ has locally trivial foliation, and its leaf space is homeomorphic to that of
  $\mathrm{Gr}(\pi_M) \times \mathrm{Gr}(\pi_X)$.
 \end{proposition}

Observe that condition b) is satisfied for example when G is
abelian and the action of G on X admits an (infinitesimal) moment map.

 \begin{proof}
 First observe that, because orbits of $G\curvearrowright X$ are tangent to the symplectic leaves of $(X,\pi_X)$, there is a canonical singular foliation $\mathcal{S}_{\Sigma}$ on $\Sigma$ induced by the singular foliations $\mathcal{S}_{M}$ on $M$ and $\mathcal{S}_{X}$ on $X$, determined by the Poisson structures $\pi_M$ and $\pi_X$: if $g_{ij}:U_i \cap U_j \to G$ is a cocycle representing $P \to M$, its composition with the action homomorphism $\psi:G \to \mathrm{Diff}(X)$ is a cocycle representing $\Sigma \to M$, and the product singular foliations $(U_i,\mathcal{S}_{M}) \times (X,\mathcal{S}_{X})$ descend under the identifications
 \begin{align*}
  & \phi_{ij}: U_i \cap U_j \times X \stackrel{\simeq}{\rmap} U_i \cap U_j \times X, & \phi_{ij}(y,x) = (y,\psi(g_{ij}(y))x)
 \end{align*}to a singular foliation $\mathcal{S}_{\Sigma}$ on $\Sigma$, which is locally trivial by its very construction, and which is independent of the choice of cocycle. Observe moreover that the hypothesis that the action of $G$ is tangent to leaves of $\mathcal{S}_X$ implies that the cocycle $(U_i,\phi_{ij})$ induces the identity map
 \begin{align*}
  U_i \cap U_j \times X/\mathcal{S}_X \stackrel{\simeq}{\rmap} U_i \cap U_j \times X/\mathcal{S}_X,
 \end{align*}where $X/\mathcal{S}_{X}$ of $\mathcal{S}_{X}$ --- that is, the topological space obtained from $X$ by identifying points which lie in the same leaf, equipped with the quotient topology. Therefore
 \begin{align*}
  \Sigma/\mathcal{S}_{\Sigma} \simeq M/\mathcal{S}_M \times X/\mathcal{S}_X.
 \end{align*}
  We claim that, under either of the hypotheses a) or b) above, the singular foliation on $\Sigma$ corresponding to the Poisson structure $\pi_{\Sigma}$ coincides with the locally trivial singular foliation constructed in the preceding paragraph. The key observation is that it suffices to check this claim for coupling submersions:
 \vspace{0.2cm}
 
 \noindent \emph{The preimage under $\Sigma \to M$ of a leaf of $\pi_M$ is saturated for both $\mathcal{S}_{\Sigma}$ and $\pi_{\Sigma}$.} For each $y \in P$, the preimage $P|_{\mathrm{S}_M(py)}=p^{-1}\mathrm{S}_M(py)$ of the symplectic leaf of $(M,\pi_M)$ through $py$ is a Poisson submanifold of $(P,\pi_P)$, and
 \begin{align}\label{eq: preimage submanifold} 
 p:(P|_{\mathrm{S}_M(py)},\pi_{P|_{\mathrm{S}_M(py)}}) \rmap (\mathrm{S}_M(py),\omega_{\mathrm{S}_M(py)})
\end{align}is a coupling Poisson submersion. It is also a principal $G$-bundle equipped with a $G$-invariant Poisson structure, and its associated bundle
 \begin{align}\label{eq: associated submanifold} 
 & p:(\Sigma|_{\mathrm{S}_M(py)},\pi_{\Sigma|_{\mathrm{S}_M(py)}}) \rmap (\mathrm{S}_M(py),\omega_{\mathrm{S}_M(py)}), & \Sigma|_{\mathrm{S}_M(py)} = P|_{\mathrm{S}_M(py)}\times_GX
\end{align}is the preimage of $\mathrm{S}_M(py)$ under $p:\Sigma \to M$ --- which (again by Theorem \ref{thm : restriction of PD submersion over leaf}) is a Poisson submanifold of $(\Sigma,\pi_{\Sigma})$. 

\vspace{0.2cm}

As a consequence, in order the prove the proposition, it suffices to verify that:

\vspace{0.2cm}

\noindent \emph{If $\pi_M$ is symplectic, then $\pi_{\Sigma}$ induces $\mathcal{S}_{\Sigma}$.} For in that case $p:(P,\pi_P) \to (M,\pi_M)$ is a coupling Poisson submersion, which, being $G$-invariant, is strongly locally trivial (as in Example \ref{ex: complete}). By Remark \eqref{rem : principal trivializations come for free in the symplectic case}, a local trivialization
\begin{align*}
 \phi : (U,\pi_M) \times (G,\pi_G) \rmap (P,\pi_{P})
\end{align*}of (\ref{eq: associated submanifold}) induces a local trivialization
\begin{align*}
 \overline{\phi} : (U,\pi_M) \times (X,\widetilde{\pi}_X) \rmap (\Sigma,\pi_{\Sigma}),
\end{align*}where $\widetilde{\pi}_X:=\wedge^2\rho(\pi_{G,e}) + \pi_X$ in the notation of (\ref{eq: new fiber factor}). 
 The proof concludes with the observation that, under either assumption a) or b) in the statement, $\widetilde{\pi}_X$ and $\pi_X$ induce the same singular foliation $\mathcal{S}_{X}$ on $X$.
For a) this is straightforward, whereas for b) one still needs to argue that $\rho(\mathfrak{g})\subset \widetilde{\pi}_X^{\sharp}(T^*X)$. But for any $v\in \mathfrak{g}$ and any $x\in X$, since the orbit $G\cdot x$ is isotropic  we have
\begin{align*}
 & \rho(v)_x=\pi_X^{\sharp}(\xi_x), & \xi_x\in N^*_x(G\cdot x),
\end{align*}
 which gives $\widetilde{\pi}_X^{\sharp}(\xi_x)=\pi_X^{\sharp}(\xi_x)$, since $\wedge^2\rho(\pi_{G,e})$ vanishes on $N^*_x(G\cdot x)$. Therefore
\begin{align*}
 (U_i,\pi_M) \times (X,\widetilde{\pi}_X) \qquad \text{defines} \qquad \phi^*\mathcal{S}_{\Sigma} = \mathcal{S}_{M}|_{U_i} \times \mathcal{S}_{X}.
\end{align*} 
Hence $\mathcal{S}_{\Sigma}$ is the singular foliation induced by $\pi_{\Sigma}$. 
 \end{proof}

\subsection{Poisson structures with finitely many leaves.}

As a final application of our methods, we construct associated bundles in which the Poisson structure on the total space has a finite number of symplectic leaves. The building blocks of our construction are two classes of Poisson manifolds with finitely many leaves:

\begin{enumerate}[i)]
 \item \emph{toric Poisson manifolds} (coming from a non-degenerate positive bivector as described in Proposition \ref{pro : positive GIT}),  whose symplectic leaves are orbits of the action of a complex torus;
 \item \emph{manifolds of full flags} (as described in Proposition \ref{pro : Lu-Weinstein}), whose symplectic leaves are Bruhat cells \cite{LW};
\end{enumerate}

\begin{proposition}\label{pro: bruhat base finite} Let $(G,\pi_G)$ be a compact, connected semisimple Lie group with its ``standard'' Poisson structure corresponding to the maximal torus $T \subset G$, and let $(X,\pi_X)$ be either:

--- a manifold of full flags $G'/T'$, again with its ``standard'' Poisson structure, or

--- a toric Poisson manifold $T'_\C \curvearrowright M_\Delta$.

Then any group homomorphism $T \to T'$ determines on the associated bundle $\Sigma=G\times_T X$ a Poisson structure $\pi_\Sigma$  with a finite number of symplectic leaves.
\end{proposition}
\begin{proof}
By Proposition \ref{pro : Lu-Weinstein} the Poisson submersion $(G,\pi_G)\to (G/T,\pi_{G/T})$ has Poisson fibres with the trivial Poisson structure. We argue that the orbits of the torus action on both manifolds of full flags and toric varieties lie inside symplectic leaves --- and thus fall within the hypotheses of Proposition \ref{pro : zero or isotropic fiber associated}. This in particular implies that the leaf space of the induced Poisson structure $\pi_{\Sigma}$ on the ensuing associated bundle $\Sigma$ is homeomorphic to the product of the leaf space of $(M,\pi_X)$ and that of $(X,\pi_X)$.

--- For a manifold of full flags $G'/T'$, the left action of $T'$ on $(G',\pi_G')$ is by Poisson diffeomorphisms (since inversion on $G$ is an anti-Poisson map), and since $(G'/T',\pi_{G'/T'})$ has finitely many leaves, the orbits of  $T' \curvearrowright G'/T'$ lie inside symplectic leaves.

--- For a toric Poisson manifold $T' \curvearrowright M_\Delta$, the action of $T'_\C$ on $M_\Delta$ is by Poisson diffeomorphism, and its orbits are exactly the symplectic leaves of the Poisson structure.
\end{proof}

\begin{example}\label{ex : flag base}\normalfont
Consider the Poisson-Lie group $(\mathrm{SU}_2,\pi_{\mathrm{SU}_2})$ corresponding to $\mathrm{SO}_2$ and the toric Poisson manifold $(\mathbb{C}P^{1},\pi_{\mathbb{C}P^{1}})$ arising from $\mathbb{C}^{2}\diagdown \{0\}$. Then a degree $k \in \mathbb{Z}$ homomorphism $z \mapsto z^k$ determines Poisson submersions with Poisson fibres
\begin{align*}
 & \mathrm{SU}_2 \times_{\mathrm{SO}_2} (\mathrm{SU}_2/\mathrm{SO}_2) \to \mathrm{SU}_2/\mathrm{SO}_2, & \mathrm{SU}_2 \times_{\mathrm{SO}_2} \mathbb{C}P^{1} \to \mathrm{SU}_2/\mathrm{SO}_2
\end{align*}which have locally trivial foliation (and have respectively four and six leaves).
\end{example}

\begin{proposition}\label{pro: toric base finite} Let $p:(P_\Delta,\pi_P)\to (M_\Delta,\pi_M)$ be 
the ($N_\C$-principal) GIT presentation of the Poisson toric manifold $M_\Delta$, and let $(X,\pi_X)$ be either:

--- a manifold of full flags $G'/T'$, or

--- a toric Poisson manifold $T'_\C \curvearrowright M'_\Delta$ whose Poisson structure is induced by a totally real bivector. 

Then a group homomorphism $N \to T'$ determines
on the associated bundle $\Sigma=P_\Delta\times_N X$
a Poisson structure $\pi_{\Sigma}$  with a finite number of symplectic leaves.
\end{proposition}
\begin{proof}
To apply Proposition \ref{pro : zero or isotropic fiber associated}, it suffices to show that the orbits of $T'\curvearrowright (X,\pi_X)$ are isotropic. (Note that the fibres of $p:(P_\Delta,\pi_P)\to (M_\Delta,\pi_M)$ are symplectic). 

--- For manifolds of full flags this is a consequence of the (global symplectic) charts in \cite{Lu} for the Bruhat cells: 
the symplectic form splits as a product of area forms and the torus action is induced from a Cartesian product of rotations in the plane. 

--- For toric Poisson manifolds, this is a consequence of the fact that the positive, nondegenerate bivector of which it is a quotient is totally real, and therefor the fixed-point set of complex conjugation --- that is, the Lie algebra of $T'$ --- is Lagrangian. 
\end{proof}

\begin{example}\normalfont
Consider the Poisson-Lie group $(\mathrm{SU}_2,\pi_{\mathrm{SU}_2})$ corresponding to $\mathrm{SO}_2$ and a totally real toric Poisson manifold $(\mathbb{C}P^{1},\pi_{\mathbb{C}P^{1}})$ arising from $\mathbb{C}^{2}\diagdown \{0\}$. Then a degree $k \in \mathbb{Z}$ homomorphism $z \mapsto z^k$ determines Poisson submersions with Poisson fibres
\begin{align*}
 & (\mathbb{C}^{2}\diagdown \{0\}) \times_{\mathrm{SO}_2} (\mathrm{SU}_2/\mathrm{SO}_2) \to \mathbb{C}P^{1}, & (\mathbb{C}^{2}\diagdown \{0\}) \times_{\mathrm{SO}_2} \mathbb{C}P^{1} \to \mathbb{C}P^{1}
\end{align*}which have locally trivial foliation (and have respectively six and nine leaves). The diffeomorphism type of these associated bundles (and those in Example \ref{ex : flag base}) is determined by the parity of $k$, and it would be interesting to understand in which cases the Poisson structures with the same number of
symplectic leaves are Poisson diffeomorphic.  
\end{example}

\section{Appendix}

Symplectic leaves of a Poisson manifold $(M,\pi)$ are not in general embedded submanifolds. However, because of the Weinstein splitting, the following is true: for each $x \in M$, there exists a diffeomorphism
\begin{align*}
    \varphi : U \diffto V \times W
\end{align*}
where $U$ is an open, connected neighborhood $U$ of $x$ in $M$, $V$ is an open, connected neighborhood of $x$ in $\mathrm{S}_M(x)$, and $x \in W \subset M$ is a connected submanifold transverse to $V$, with the property that
\begin{align*}
    & \varphi|_{V}=\mathrm{id}_V, & \varphi|_{W}=\mathrm{id}_W, 
\end{align*}and, for each leaf $\mathrm{S}_M(y)$ of $\pi$,
\begin{align*}
    & \varphi(U \cap \mathrm{S}_M(y)) = V \times \Lambda(y), & \Lambda(y) \subset W,
\end{align*}with $\Lambda(x)$ being at most countable. This suggests the notion of a \emph{leaf-like submanifold} (see \cite[Appendix B]{PT1.875}), which includes as examples leaves of singular foliations or Lie algebroids:

\begin{definition}
    A subset $S$ of a smooth manifold $M$ is a {\bf leaf-like submanifold} if, around each point $x \in S$ there is an open neighborhood $U$ of $x$ in $M$, together with a diffeomorphism \[\phi : U \diffto V \times W\]
    from $U$ into the product of connected manifolds $V$ and $W$, such that \[\phi(U \cap S) = V \times \Lambda\]
    for a subset $\Lambda \subset W$ which is at most countable.
\end{definition}
The submanifolds $S^U_w:=\phi^{-1}(V \times \{w\}) \subset U \cap S$, as $w$ ranges in $\Lambda$, are called the {\bf plaques} of $S$ over $U$. Each plaque is an embedded submanifold, and plaques partition $U \cap S$:
\begin{align*}
        U \cap S = \coprod_{w \in \Lambda} S^U_w.
    \end{align*}

As explained in \cite[Appendix B]{PT1.875}, leaf-like manifolds are \emph{initial} submanifolds --- that is, they are abstract manifolds equipped with an injective immersion $j:S \to M$, with the property that, if $f:N \to M$ is a smooth map whose image lies in $S$, then the unique set-theoretic map $\widetilde{f}:N \to S$ through which $f$ factors is smooth. This implies that the differentiable structure on $S$ in uniquely determined by that of $M$.

\begin{definition}\label{def : clean}
    Let $X \subset M$ be an embedded submanifold, and $S \subset M$ be a leaf-like submanifold. We say that $X$ and $S$ {\bf intersect cleanly} if 
\begin{enumerate}[a)]
 \item $X \cap S$ is an embedded submanifold of $S$;
 \item $T(X \cap S)=TX \cap TS$.
\end{enumerate}
\end{definition}

\begin{remark}\label{rem : second countable}\normalfont
    Condition a) in Definition \ref{def : clean} should be clarified. Because manifolds are supposed to be second-countable, they can have at most countably many connected components. Hence, when we say $X \cap S$ is a submanifold, it is implied that $X \cap S$ is the disjoint union of \underline{countably many} connected (second-countable) submanifolds. However, we do allow connected components to have \underline{different dimensions} (as in Example \ref{ex : helicoid}).
    \end{remark}


This definition recovers the notion of clean intersection of manifolds \cite[Appendix C]{Hor} when $S$ is also embedded. 

\begin{remark}\normalfont
    In Definition \ref{def : clean}, the notions of embedded- and leaf-like submanifolds play asymmetric roles, in that we do not require that the intersection $X \cap S$ be a leaf-like submanifold of $X$. As the example below shows, that need not always be the case.
\end{remark}

\begin{example}\label{ex : helicoid}\normalfont Let $M$ be the 3-torus endowed with a  Kronecker-type foliation, meaning that every one-dimensional leaf is dense. Select a leaf $S$ and fix a foliated chart $\varphi:U\to \R^3$ so that the vertical axis corresponds to a plaque $S_i$ belonging to $S$. Define $X\subset U$ to be the preimage by $\varphi$ of the helicoid in $\R^3$ with axis the vertical axis:
\begin{align*}
    & \mathbb{R}^2 \to \mathbb{R}^3, & (t,s) \mapsto (t\cos(s),t\sin(s),s)\end{align*}
Then the intersection of $X$ and $S$ is clean and consists of $S_i$ and a countable collection of points which accumulate in $S_i$. Therefore $X\cap S$ is not a leaf-like submanifold.
\end{example}

In contrast, under Definition \ref{def : clean}, a clean intersection of an embedded submanifold with a leaf-like submanifold is always an initial submanifold of the former:

\begin{lemma}\label{lem : app : leaf like clean is initial}
    If an embedded submanifold $X \subset M$ intersects a leaf-like submanifold $S \subset M$ cleanly, then $X \cap S$ is an initial submanifold of $X$.
\end{lemma}
\begin{proof}
    By hypothesis, $X \cap S$ is an embedded submanifold of $S$. Denote by
\begin{align*}
    && i:X \to M, && j:S \to M, && i':X \cap S \to S, && j':X \cap S \to X
\end{align*}the implied immersions. Because $S$ is leaf-like, $j$ is initial, and because $X$ is embedded, $i$ and $i'$ are initial. Consider a smooth map $f:N \to X$, whose image lies inside $X \cap S$. Then the image of the composition $i \circ f:N \to M$ with the inclusion $i:X \to M$ lies inside the image of the inclusion $j:S \to M$, and the latter, being an initial submanifold, has a smooth lift $\widetilde{f}:N \to S$, such that
\begin{align*}
    i \circ f = j \circ \widetilde{f}.
\end{align*}The image of $\widetilde{f}$ is contained in the embedded submanifold $X \cap S$, which is also initial, and therefore $\widetilde{f}$ has a smooth lift $\widehat{f}:N \to X \cap S$, with
\begin{align*}
    \widetilde{f} = i' \circ \widehat{f}.
\end{align*}
Hence the lift $F:=j'\circ \widehat{f}$ of $f$ is smooth:
\[
\xymatrix{
 X \cap S \ar[dd]_{j'} \ar[rr]^{i'} & & S \ar[dd]^j\\
 & N \ar[dl]^f \ar[ur]_{\widetilde{f}} \ar[ul]_{\widehat{f}}\\
 X \ar[rr]_i & & M
}
\]
and this shows that $j'$ is initial as well.
\end{proof}

\begin{proposition}\label{prop:app-clean}
     Let $X$ be an embedded submanifold, and $S$ be a leaf-like submanifold of a smooth manifold $M$. Then the following assertions are equivalent:
\begin{enumerate}[i)]
\item $X$ and $S$ intersect cleanly.
 \item $X \cap S$ is a disjoint union of  (a priori uncountably many) initial submanifolds $Z$, for which
 \[T_zZ=T_zX\cap T_z S\]
 for each $z\in Z$. 
\item Every $z\in X\cap S$ is the center of a coordinate chart $\varphi : M \supset U \to \R^m$ which is adapted to $X$ and to the plaque of $S$ over $U$ which passes through $z$.
\end{enumerate}
\end{proposition}
\begin{proof}
\noindent \emph{i) implies ii)} By definition of clean intersection, if $X$ and $S$ meet cleanly, then
\begin{align*}
    X \cap S = \coprod Z_i
\end{align*}is a disjoint union of at most countably many embedded, connected submanifolds $Z_i \subset S$, with
\begin{align*}
    & T_zZ_i = T_zX \cap T_zS, & z \in Z_i.
\end{align*}Hence i) (trivially) implies ii).\\

\noindent \emph{ii) implies iii)} Assume now that
\begin{align*}
    X \cap S = \coprod Z_i
\end{align*}is a disjoint union of \emph{possibly uncountably many} connected, initial submanifolds $Z_i \subset M$, with
\begin{align*}
    & T_zZ_i = T_zX \cap T_zS, & z \in Z_i.
\end{align*}
Because $S$ is leaf-like, we can find an open neighborhood $U$ of any $z \in Z_i$, for which
\begin{align*}
    U \cap S = \coprod S_a,
\end{align*}
is a disjoint union of at most countably many plaques. Write $U \cap Z_i$ as a countable disjoint union of its connected components:
\begin{align*}
    U \cap Z_i = \coprod_{A(i)} W_{\beta}.
\end{align*}Then $U \cap X \cap S_a$ is a union of such connected components. Let $S_0$ denote the plaque of $S$ over $U$ which passes through $z$, and $W_0$ the connected component of $U \cap X \cap S_0$ through $z$. Then $T_zW_0=T_zX \cap T_zS_0$ ensures that one can build as in \cite[Proposition C.3.1]{Hor} a coordinate chart $(U,\varphi)$ of $M$ which is centered at $z$ and is adapted to both $X$ and $S_0$. So ii) implies iii). \\

\noindent \emph{iii) implies i)} If around each $z \in X \cap S$ one can find a chart $\varphi : M \supset U \to \mathbb{R}^m$ which is centered at $z$, and
\begin{align*}
    & \varphi(U \cap X) = \varphi(U) \cap W_X, & \varphi(U \cap S_0) = \varphi(U) \cap W_S
\end{align*}where $S_0$ denotes the plaque of $S$ over $U$ through $z$, and $W_X$ and $W_S$ denote vector subspaces of $\mathbb{R}^m$, then
\begin{align*}
    & \varphi(U \cap X \cap S_0) = \varphi(U) \cap W_X \cap W_S
\end{align*}shows that 
\begin{align*}
    \varphi|_{U \cap S_0}:U \cap S_0 \to W_S
\end{align*}is a coordinate chart of $S$ adapted to $U \cap X \cap S_0$. This shows that $X \cap S$ is an embedded submanifold of $S$, and moreover,
\begin{align*}
    && z \in X \cap S && \Longrightarrow && T_z(X \cap S) = T_zX \cap T_zS. 
\end{align*}Therefore iii) implies i).
\end{proof}

\begin{remark}\normalfont
The following subtlety in the formulation of Proposition \ref{prop:app-clean} is worth mentioning : we raised in item ii) the \emph{a priori} possibility that $X\cap S$ be the disjoint union of \emph{uncountably} many submanifolds --- and therefore \emph{not} itself a submanifold, see Remark \ref{rem : second countable} --- only to rule out that possibility by concluding that $X \cap S$ must in fact be an embedded submanifold of $S$. 
\end{remark}


\end{document}